\documentclass{amsart}

\usepackage{amssymb}

\usepackage{amsthm}

\usepackage{amsmath}
\usepackage{epstopdf}

\usepackage[font=small]{caption}
\usepackage[labelformat = empty,position=top]{subcaption}
\usepackage[export]{adjustbox}
\usepackage{pinlabel}
\usepackage{color}
\usepackage[a4paper, left=3cm, right=2cm]{geometry}

\newtheorem{defi}{Definition}[section]

\newtheorem{lemma}[defi]{Lemma}
\newtheorem{proposition}[defi]{Proposition}
\newtheorem{corollary}[defi]{Corollary}
\begin{document}

%\begin{frontmatter}

\title{Crossing numbers of composite knots and spatial graphs}
\author{Benjamin Bode}

\address{H H Wills Physics Laboratory, University of Bristol, Bristol BS8 1TL, UK,\newline
present address: Department of Mathematics, Osaka City University, Sugimoto, Sumiyoshi-ku, Osaka 558-8585, Japan}
\email{ben.bode.2013@my.bristol.ac.uk}

\maketitle
\begin{abstract}
We study the minimal crossing number $c(K_{1}\# K_{2})$ of composite knots $K_{1}\# K_{2}$, where $K_1$ and $K_2$ are prime, by relating it to the minimal crossing number of spatial graphs, in particular the $2n$-theta-curve $\theta_{K_{1},K_{2}}^n$ that results from tying $n$ of the edges of the planar embedding of the $2n$-theta-graph into $K_1$ and the remaining $n$ edges into $K_2$. We prove that for large enough $n$ we have $c(\theta_{K_1,K_2}^n)=n(c(K_1)+c(K_2))$. We also formulate additional relations between the crossing numbers of certain spatial graphs that, if satisfied, imply the additivity of the crossing number or at least give a lower bound for $c(K_1\# K_2)$.
\end{abstract}

%\begin{keyword}

%crossing number \sep composite knots \sep spatial graphs \sep theta-curves

%\MSC 57M25
%\end{keyword}

%\end{frontmatter}


\section{Introduction}\label{sec:intro}
It is one of the oldest open conjectures in knot theory that the minimal crossing number is additive under the connected sum operation. That is, given two knots $K_{1}$ and $K_{2}$ of minimal crossing numbers $c(K_{1})$ and $c(K_{2})$ respectively, is it true that $c(K_{1}\# K_{2})=c(K_{1}+K_{2})$?
A positive answer to this question would not only help the understanding of this most fundamental knot invariant, but also contradict other conjectures, for example that the percentage of hyperbolic knots among all prime knots of minimal crossing number at most $n$ approaches 100 as $n$ goes to infinity \cite{malyutin}.

By definition of the connected sum (cf. Figure \ref{fig:def}), we have $c(K_{1}\# K_{2})\leq c(K_{1})+c(K_{2})$. Equality is established if both knots are torus knots \cite{diao, torus} or if both are alternating \cite{murasugi, kaufalt, thistle} (or more general adequate \cite{lt88}), but in general it is not even known if $c(K_{1}\# K_{2})\geq c(K_{1})$. The best lower bound that we are aware of, $c(K_{1}\# K_{2})\geq \tfrac{1}{152} (c(K_{1})+c(K_{2}))$, was shown by Lackenby \cite{lackenby}. In fact, he showed the stronger result that $c(K_{1}\# K_{2}\#\ldots\# K_{n})\geq \tfrac{1}{152} \sum_{k=1}^{n}c(K_{i})$ for all knots $K_{i}$.

In this paper we prove relations between the minimal crossing numbers of composite knots and certain spatial graphs, in particular theta-curves. We also formulate additional relations that, if satisfied, imply the additivity of crossing numbers or at least give a lower bound. Checking these conditions is very challenging, but we hope that this work inspires a general method to make progress in the crossing number conjecture.

A theta-curve is an embedding of the theta-graph $\theta$ (cf. Figure \ref{fig:def1}a)) in $S^{3}$, the planar graph consisting of two vertices with three edges between them. Theta-curves are studied up to equivalence under ambient isotopy. Therefore a large number of tools from knot theory applies to the theory of theta-curves as well.
In particular, we can study theta-curves by considering their diagrams, projections in the plane with at most double points at which intersections are transverse.

Thus many diagrammatic invariants that were defined to distinguish knots and links, such as the minimal crossing number, extend to theta-curves. We label the edges of a theta-curve by $x$, $y$ and $z$ as in Figure \ref{fig:def1}a) and denote the numbers of crossings between two strands, by the concatenation of the two corresponding letters. Hence $xy$ denotes the number of crossings between the $x$-strand and the $y$-strand, $xx$ denotes the number of crossings of the $x$-strand with itself and so on.
Theta-curves and their connections to knot theory have been studied before and especially their connections to knotoids has been stressed \cite{kauffman, turaev, knotoids}.

\begin{figure}[tb]
\centering 
\labellist
\pinlabel \textbf{a)} at -10 350
\pinlabel \textbf{b)} at 600 350
\pinlabel \textbf{c)} at 280 -20
\Large
\pinlabel ${\color{blue}K_1}$ at 170 260
\pinlabel ${\color{red}K_2}$ at 400 310
\endlabellist

%\begin{subfigure}[t]{0.03\textwidth}
%\textbf{a)}
%\end{subfigure}
\begin{subfigure}[t]{0.45\textwidth}
\includegraphics[width=\linewidth]{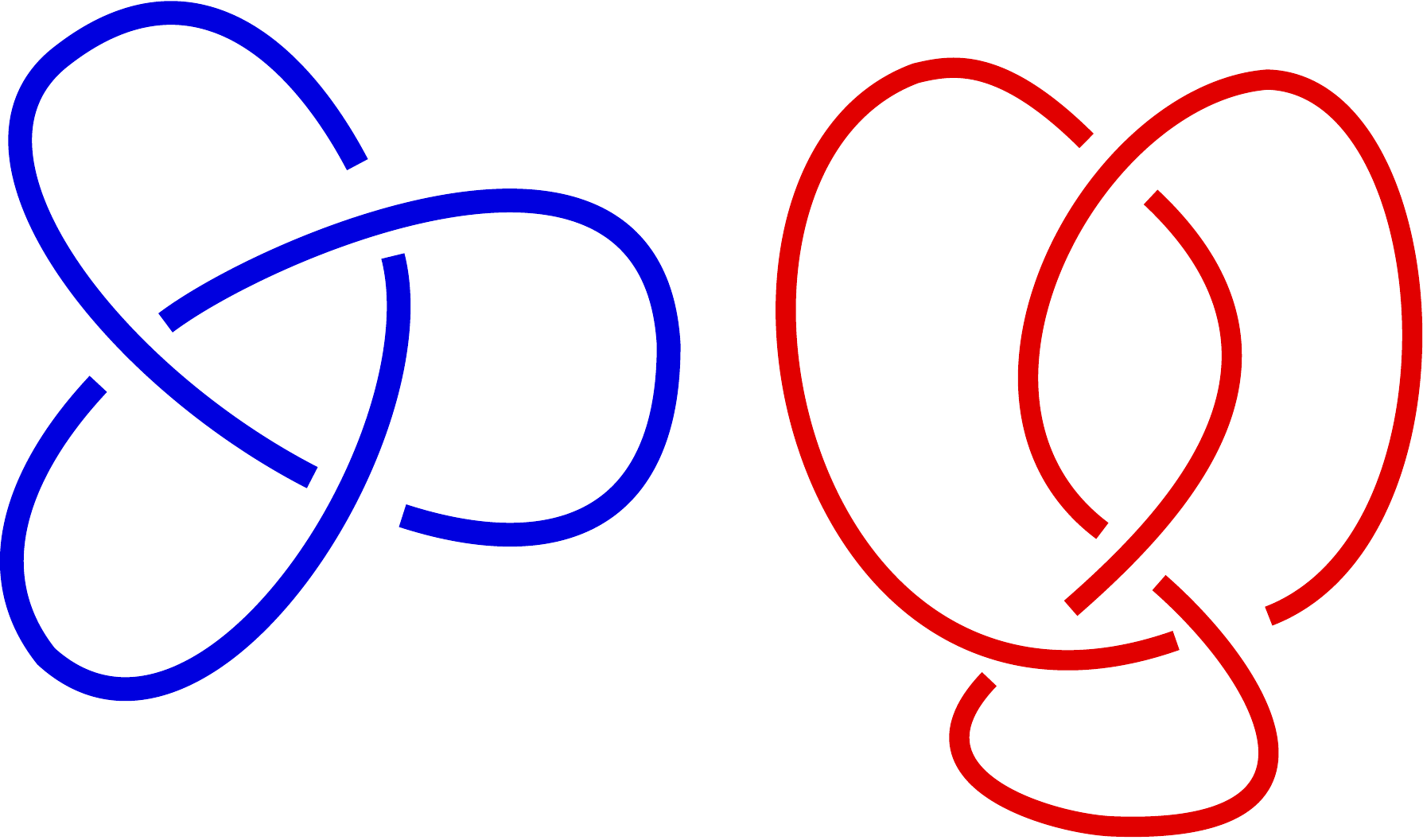}
\end{subfigure}\hfill
%\begin{subfigure}[t]{0.03\textwidth}
%\textbf{b)}
%\end{subfigure}
\begin{subfigure}[t]{0.45\textwidth}
\includegraphics[width=\linewidth, valign=b]{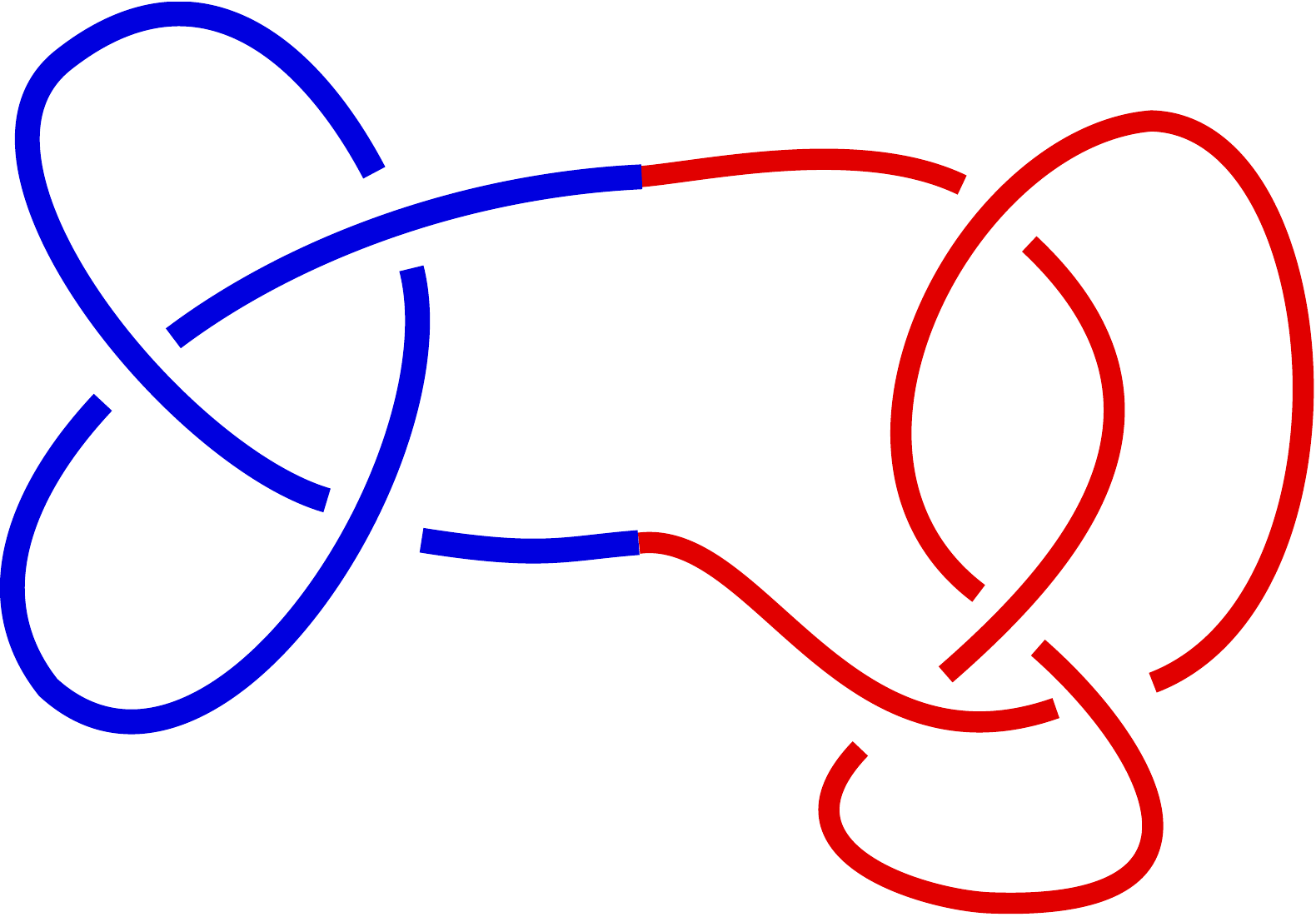}
\end{subfigure}
%\begin{subfigure}[t]{0.03\textwidth}
%\textbf{c)}
%\end{subfigure}
\begin{subfigure}[t]{0.45\textwidth}
\includegraphics[width=\linewidth]{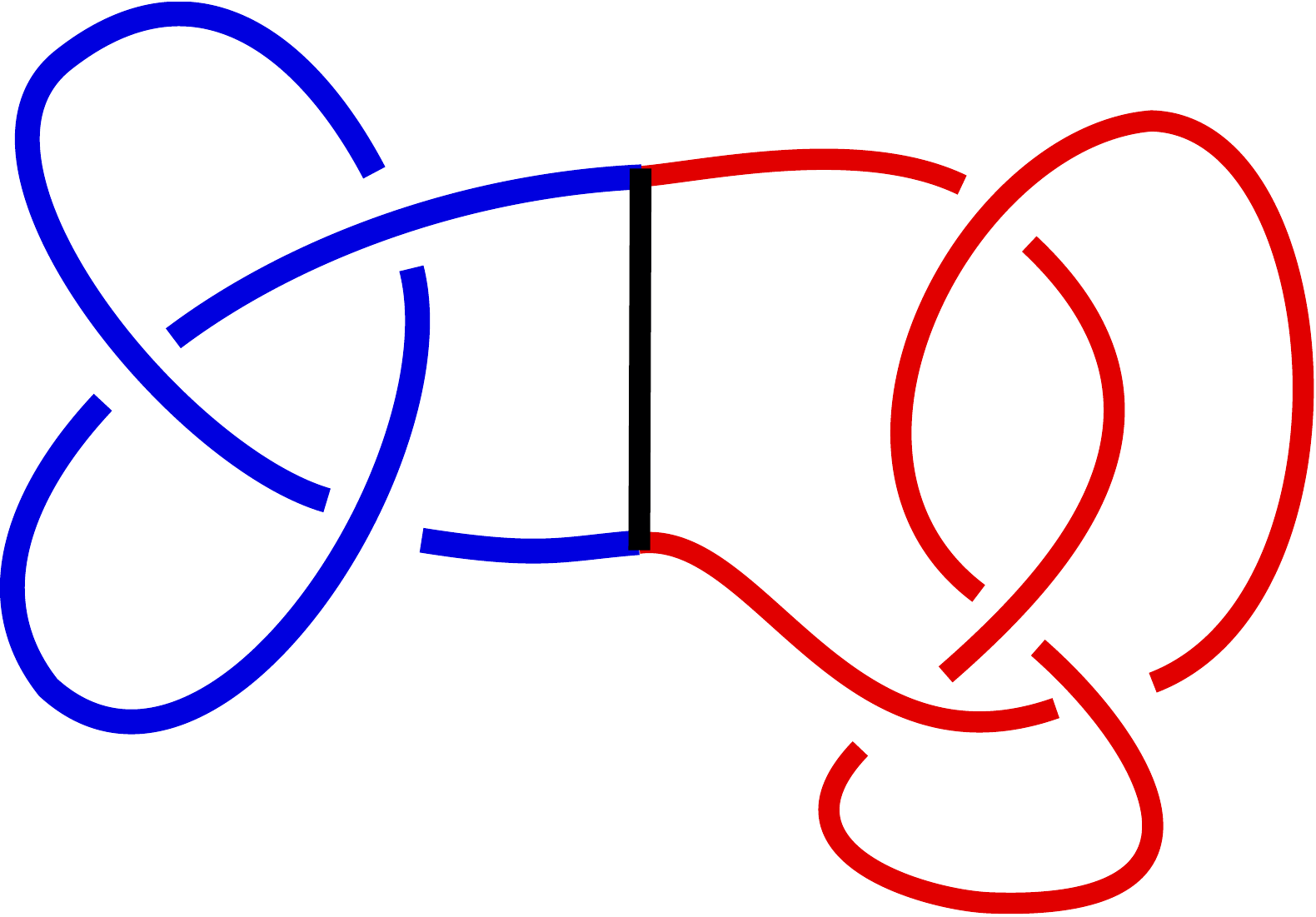}
\end{subfigure}
\caption{ a) Diagrams of a trefoil ($K_1$) and a figure eight knot $(K_2)$ b) A diagram of their connected sum $K_1\# K_2$. c) Adding an extra unknotted arc results in the theta-curve $\theta_{3_1,4_1}$.\label{fig:def}}
\end{figure}

There is a very natural way to associate a theta-curve to a pair of knots $K_{1}$, $K_{2}$, or more precisely to their connected sum $K_{1}\# K_{2}$. Consider the diagram of $K_{1}\# K_{2}$ in Figure \ref{fig:def}b) used to define the connected sum. Then adding an unknotted arc between the two points where $K_{1}$ and $K_{2}$ are glued together results in a theta-curve, denoted by $\theta_{K_1,K_2}$.
Among all theta-curves there is a unique planar embedding and we call the corresponding isotopy type the trivial theta-curve. Then $\theta_{K_{1},K_2}$ is the theta-curve that results from tying $K_{1}$ into the $x$-arc of the trivial theta-curve and $K_2$ in its $z$-arc.

Deleting any of the three edges of a theta-curve leaves a knot, in the case of $\theta_{K_1,K_2}$ we have $x\cup y=K_1$, $y\cup z=K_2$ and $x\cup z=K_1\# K_2$. Note that theta-curves are not uniquely characterised by the knot types of these three knots, their constituent knots. For example for Kinoshita's theta-curve in Figure \ref{fig:def1}b), all pairs of edges form the unknot, but it is not the planar theta-curve shown in Figure \ref{fig:def1}a). 

\begin{figure}\centering
\labellist
\pinlabel \textbf{a)} at 30 240
\pinlabel \textbf{b)} at 340 240
\endlabellist
%\begin{subfigure}[t]{0.03\textwidth}
%\textbf{a)}
%\end{subfigure}
\begin{subfigure}[t]{0.4\textwidth}
\includegraphics[width=\linewidth]{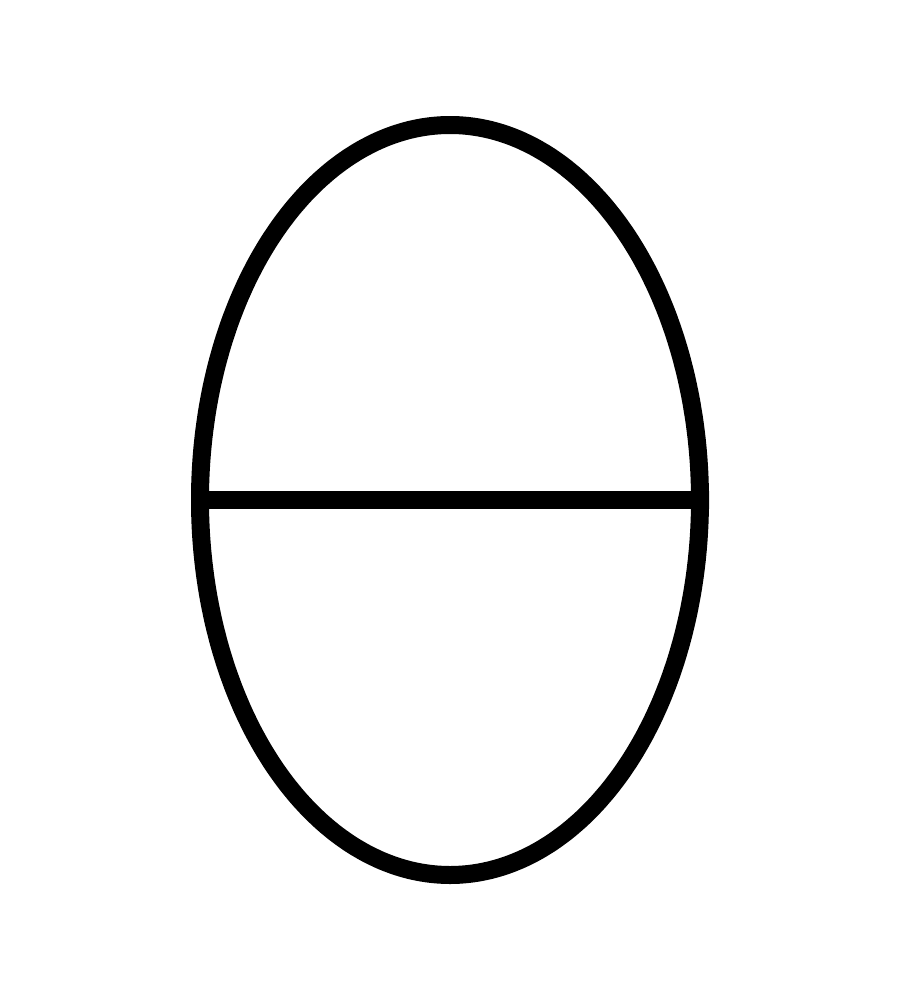}
\end{subfigure}\hfill
%\begin{subfigure}[t]{0.03\textwidth}
%\textbf{b)}
%\end{subfigure}
\begin{subfigure}[t]{0.5\textwidth}
\includegraphics[width=\linewidth, valign=b]{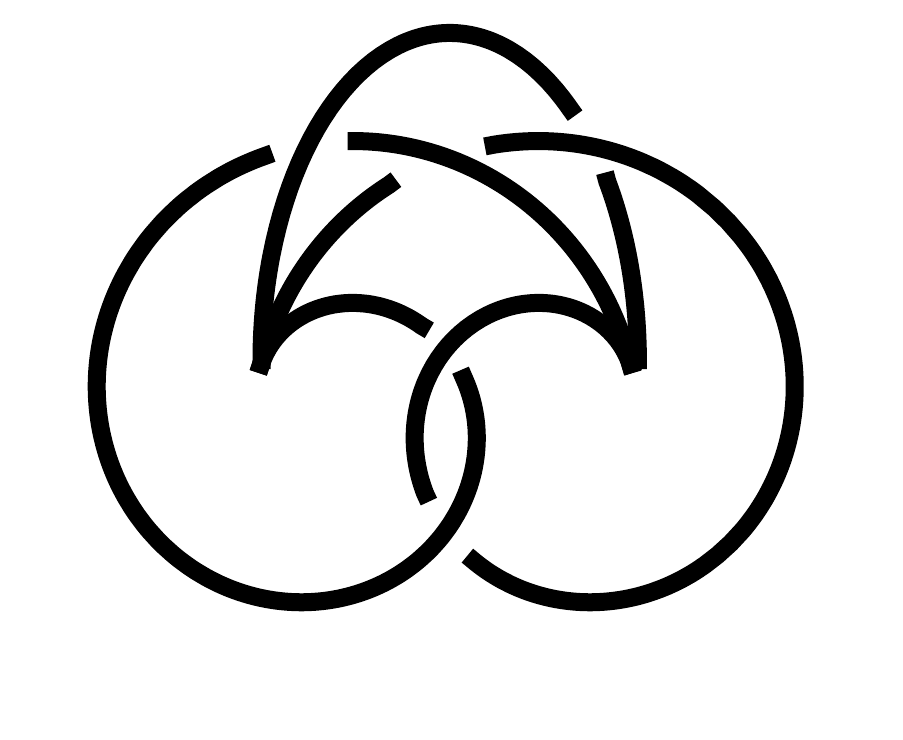}
\end{subfigure}
\caption{a) The theta-graph in its planar embedding. b) Kinoshita's theta-curve. Both theta-curves have the same constituent knots, but are not ambient isotopic.\label{fig:def1}}
\end{figure} 

Since for any diagram of $\theta_{K_1,K_2}$ we have $x\cup z=K_1\# K_2$, it is clear that $c(\theta_{K_1,K_2})\geq c(K_1\# K_2)$ and from its construction we know that $c(\theta_{K_1,K_2})\leq c(K_1)+c(K_2)$.

Although the definition of $\theta_{K_1,K_2}$ makes sense for all knots $K_1$ and $K_2$ and most statements remain true for composite knots, we require $K_1$ and $K_2$ to be prime in the following.

This paper proceeds as follows. In Section \ref{sec:proof} we relate $c(\theta_{K_1,K_2})$ to $c(K_1\# K_2)$.

In Section \ref{sec:higher} we consider theta-curves of higher degree, that is, embeddings of planar graphs with two vertices and $2n$ edges between them. We are particularly interested in embeddings, where $n$ of the edges are tied into $K_1$ and the remaining $n$ edges tied into $K_2$, similar to the case of $\theta_{K_1,K_2}$. Here we show that for large enough $n$ the minimal crossing number of these graphs is $n(c(K_1)+c(K_2))$.

Section \ref{sec:double} discusses a relation between $c(K_1\# K_2)$ and the minimal crossing numbers of the higher degree theta-curves $c(\Omega_{K_1,K_2}^n)$ that are discussed in Section \ref{sec:higher} resulting in the lower bound $c(K_1\# K_2)\geq \frac{1}{n^2}c(\Omega_{K_1,K_2}^n)$.
Thus finding values of $n$ for which $c(\Omega_{K_1,K_2}^n)=n(c(K_1)+c(K_2))$ results in a lower bound of the form $c(K_1\# K_2)\geq \frac{1}{n}(c(K_1)+c(K_2))$.

In Section \ref{sec:outlook} we discuss further spatial graphs whose crossing numbers relate to the crossing numbers of composite knots.

\section{The crossing numbers of theta-curves}
\label{sec:proof}

Consider the theta-curve $\theta_{K_1,K_2}$, which is shown in Figure \ref{fig:def}c).
Since deleting the $y$-arc in any diagram of $\theta_{K_1,K_2}$ results in a diagram of $K_{1}\# K_{2}$, we have the inequality
\begin{equation}
\label{eq:1}
xx+xz+zz\geq c(K_1\# K_2)
\end{equation}
for any diagram of $\theta_{K_1,K_2}$, where we use the notation of Section \ref{sec:intro}.

Similarly, $x\cup y=K_1$ and $y\cup z=K_2$ and we obtain
\begin{align}
\label{eq:ineq}
2c(\theta_{K_1,K_2})&=xx+xz+zz+xx+xy+yy+yy+yz+zz+xy+xz+yz\nonumber\\
&\geq c(K_1\# K_2)+c(K_1)+c(K_2)+xy+xz+yz.
\end{align}

Since $xy$, $yz$ and $xz$ are all non-negative, we obtain the inequality 
\begin{equation}
\label{eq:prop}
2c(\theta_{K_1,K_2})\geq c(K_1\# K_2)+c(K_1)+c(K_2).
\end{equation}

\begin{proposition}
\label{prop1}
The inequality in Equation (\ref{eq:prop}) is an equality if and only if $c(\theta_{K_1,K_2})=c(K_1\# K_2)=c(K_1)+c(K_2)$. 
\end{proposition}

In order to prove Proposition \ref{prop1}, we need the following lemma.

\begin{lemma}
\label{jordan}
Let $\kappa_1$, $\kappa_2$ and $\kappa_3$ be knots and let $D$ be a diagram of a theta-curve $\theta$ where $x\cup z=\kappa_1$, $y\cup z=\kappa_2$ and $x\cup y=\kappa_3$ and no pair of arcs cross each other, i.e. $xy+yz+xz=0$. Then there are knots $K_{1}'$, $K_{2}'$ and $K_{3}'$ such that $\kappa_{1}=K_{1}'\#K_{3}'$, $\kappa_{2}=K_{2}'\#K_{3}'$ and $\kappa_{3}=K_{1}'\#K_{2}'$. Furthermore, $xx\geq c(K_1')$, $yy\geq c(K_2')$, $zz\geq c(K_3')$ and thus $c(D)\geq c(K_{1}')+c(K_{2}')+c(K_{3}')$. 
\end{lemma}

\begin{proof}
Consider the diagram $D$ as a subset of the Euclidean plane with crossings as double points. Around each of the two nodes $n_{1}$, $n_{2}$ there is a neighbourhood $U(n_{i})$ such that $(U(n_{i})\backslash D)\cup\{n_{i}\}$ is path-connected. For small enough $\epsilon>0$ the boundary of the $\epsilon$-neighbourhood $U_{\epsilon}(D)=\{a\in\mathbb{R}^2\backslash (U(n_{1})\cup U(n_{2})): \min_{b\in D} |a-b|<\epsilon\}$ of $D$ is a collection of loops and  divides $\mathbb{R}^2\backslash (U(n_{1})\cup U(n_{2}))$ into a number of path-connected components.

\begin{figure}[tb]\centering
\labellist
\pinlabel \textbf{a)} at -10 280
\pinlabel \textbf{b)} at 550 280
\pinlabel \textbf{c)} at -10 -20
\pinlabel \textbf{d)} at 550 -20
\endlabellist
\begin{subfigure}[t]{0.2\textwidth}
\includegraphics[width=\linewidth]{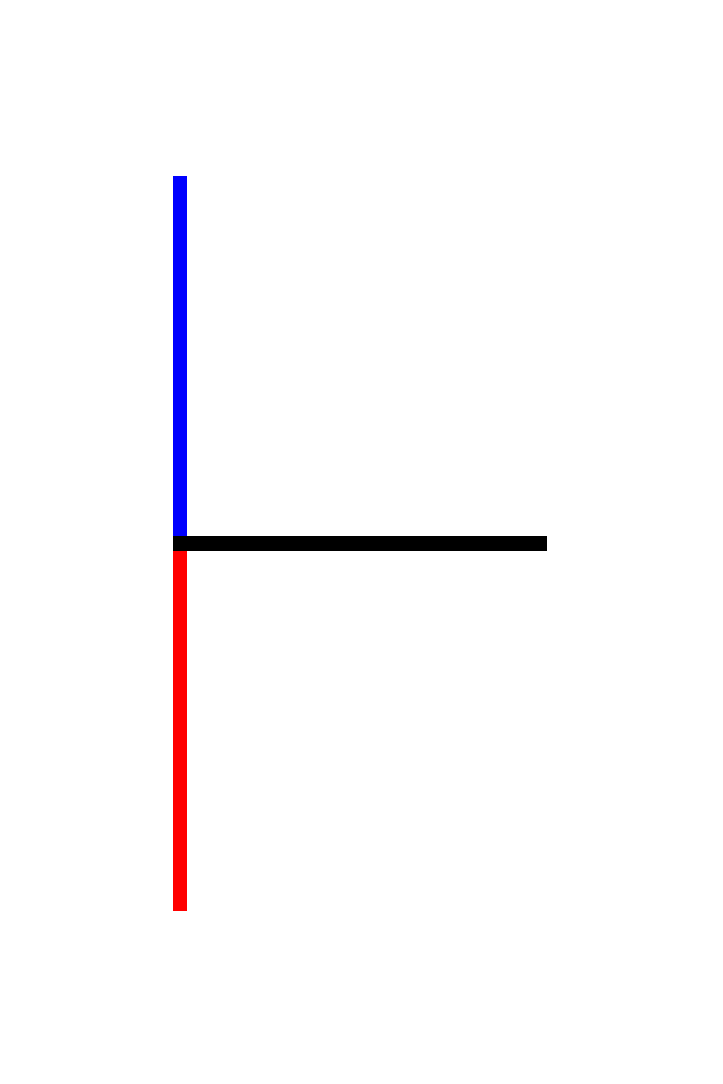}
\end{subfigure}\hfill
\begin{subfigure}[t]{0.45\textwidth}
\includegraphics[width=\linewidth, valign=b]{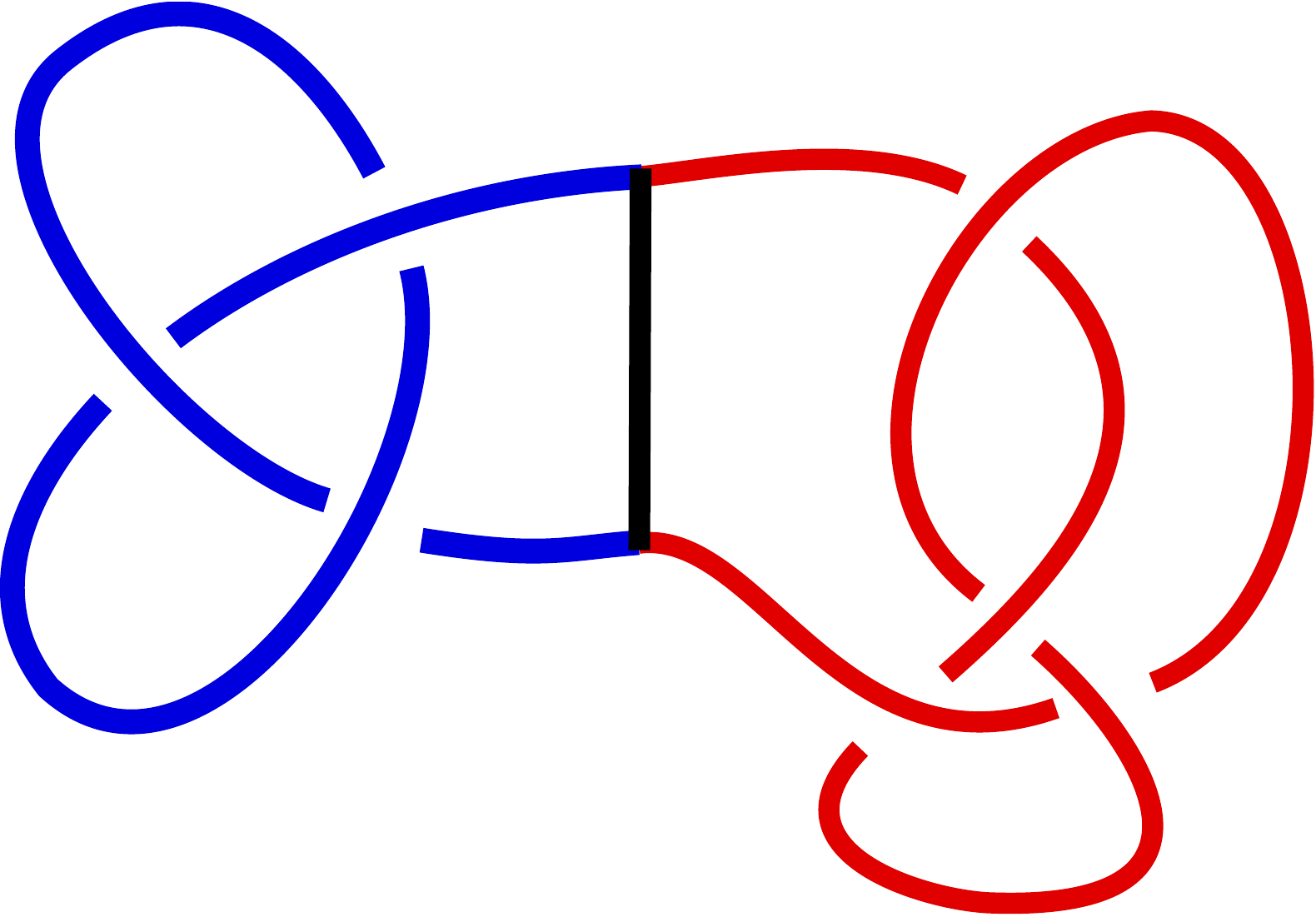}
\end{subfigure}\hfill
\begin{subfigure}[t]{0.45\textwidth}
\includegraphics[width=\linewidth, valign=b]{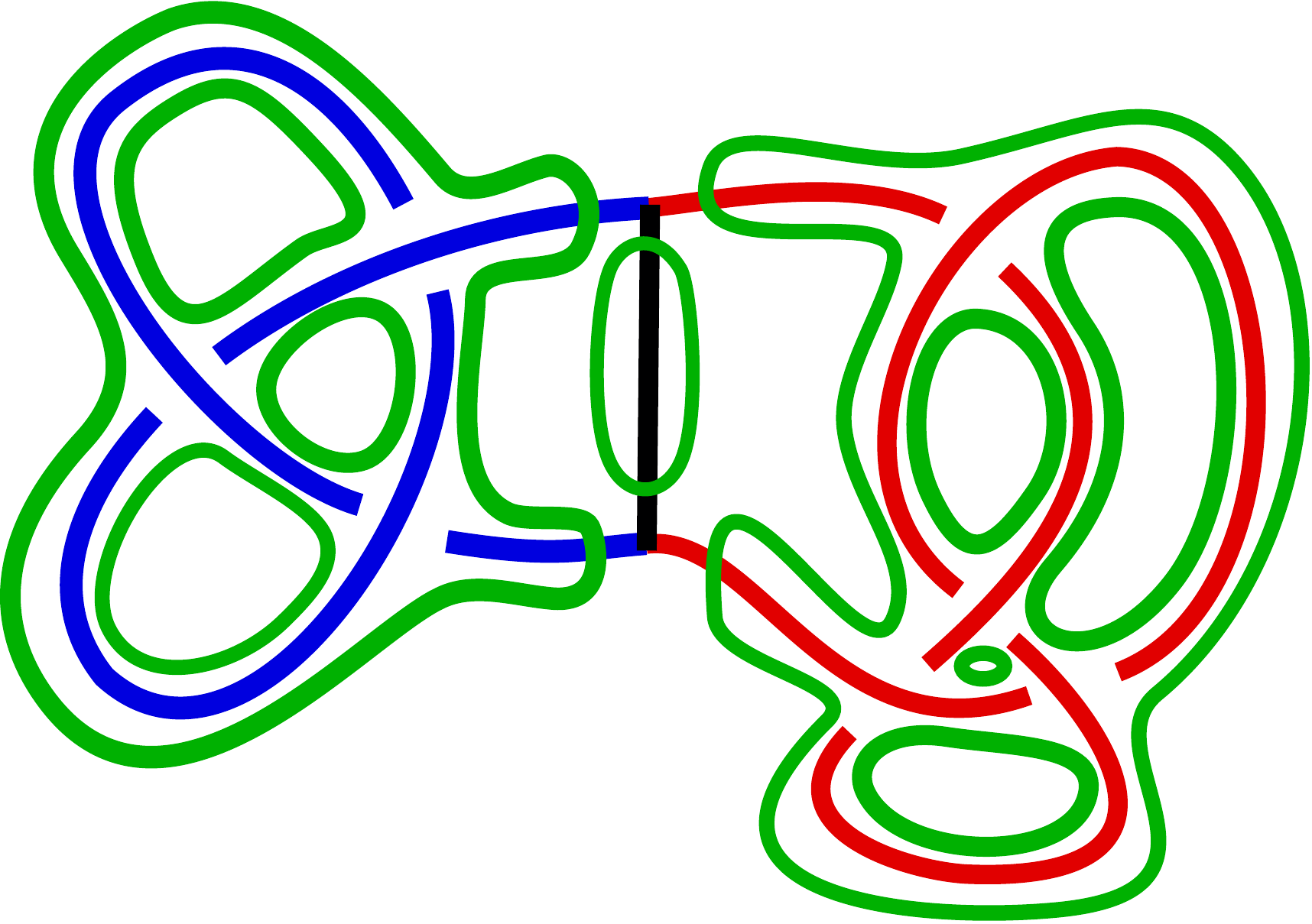}
\end{subfigure}\hfill
\begin{subfigure}[t]{0.45\textwidth}
\includegraphics[width=\linewidth, valign=b]{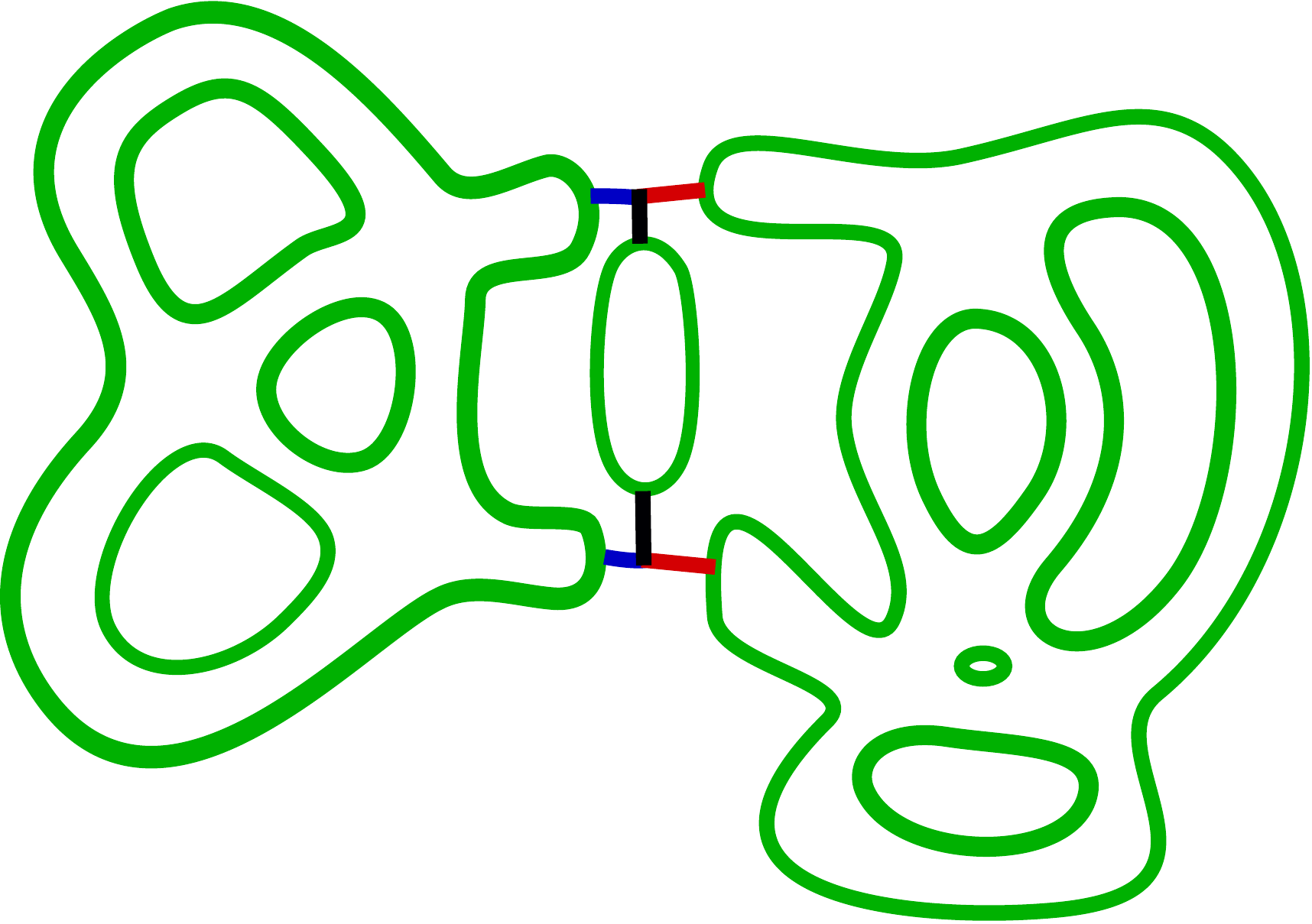}
\end{subfigure}
\caption{a) The diagram $D$ around a node $n_{i}$. b) A diagram of $\theta_{K_1,K_2}$. c) The boundaries of the $\epsilon$-neighbourhood of the diagram divide the plane into path-connected components. d) The two nodes are in the same path-connected component of $P$.\label{fig:jordan}}
\end{figure}

We claim that the two nodes are in the same component of \begin{equation}
P=(\mathbb{R}^2\backslash (\partial U_{\epsilon}(D)\cup (D\cap U(n_{1})\cup(D\cap U(n_{2}))\cup\{n_1\}\cup\{n_2\}
\end{equation}
shown in Figure \ref{fig:jordan}d). Then there is a path $\gamma\subset P$ from $n_{1}$ to $n_{2}$. Since $\gamma$ does not cross $\partial U_{\epsilon}(D)$, $D\cap U(n_1)$ or $D\cap U(n_2)$, it does not have any crossings with $D$ and it can be be chosen to not cross itself.
Call $K_{1}':=x\cup\gamma$, $K_{2}':=y\cup\gamma$ and $K_{3}':=z\cup\gamma$. Since $\gamma$ does not have any crossings with $D$ or with itself,  we have $xx+x\gamma+\gamma\gamma=xx\geq c(K_{1}')$ and similarly $yy\geq K_{2}'$ and $zz\geq K_{3}'$. Note that it follows from the uniqueness of prime decomposition of knots that $xy=xz=yz=0$ implies that $x\cup y=K_{1}'\# K_{2}'$, $y\cup z=K_{2}'\# K_{3}'$ and $x\cup z=K_{1}'\# K_{3}'$.

What is left to show is the claim that the two nodes are in the same path component of $P$.
Assume they are not in the same path component. Then there is a loop $\ell\in U_{\epsilon}(D)$ such that one of the nodes is in the bounded component of $\mathbb{R}^2\backslash \ell$ and the other one is in the unbounded component. Since $xy=yz=xz=0$, the loop $\ell$ is a boundary component of exactly one of $U_{\epsilon}(x)=\{p\in\mathbb{R}^2\backslash (U(n_{1})\cup U(n_{2})): \min_{q\in x} |p-q|<\epsilon\}$, $U_{\epsilon}(y)$ or $U_{\epsilon}(z)$ (defined analogously). But since $x$, $y$ and $z$ are paths from $n_{1}$ to $n_{2}$, all of them must cross $\ell$. Then all of them must also cross the arc associated to $\ell$ (i.e. $x$ if $\ell$ is a boundary component of $U_{\epsilon}(x)$ and so on) contradicting $xy=yz=xz=0$. This proves the claim and finishes the proof of the lemma.
\end{proof}

\begin{proof}[Proof of Proposition \ref{prop1}]
Note that in the case of $\theta=\theta_{K_1,K_2}$, we have $\kappa_1=K_1$, $\kappa_2=K_2$ and $\kappa_3=K_1\# K_2$.
 
We assume that $2c(\theta_{K_1,K_2})=c(K_{1})+c(K_{2})+c(K_{1}\# K_{2})$. Then by Equation \ref{eq:ineq} we have $xy=yz=xz=0$. Now we apply Lemma \ref{jordan} to $\theta_{K_1,K_2}$. We thus have knots $K_1'$, $K_2'$ and $K_3'$ such that $K_1=K_1'\# K_3'$, $K_1\# K_2=K_1'\# K_2'$ and $K_2=K_1'\# K_3'$.
Note that this implies $K_3'=K_1$, $K_2'=O$ and $K_3'=K_2$ and thus $c(\theta_{K_1,K_2})\geq c(K_1)+c(K_2)$. Therefore $c(\theta_{K_1,K_2})=c(K_1)+c(K_2)$ and since we assumed $c(K_{1}\# K_{2})=2c(\theta_{K_1,K_2})-c(K_{1})-c(K_{2})$, we have $c(K_1\# K_2)=c(K_1)+c(K_2)$.

Now assume that $c(\theta_{K_1,K_2})=c(K_1\# K_2)=c(K_1)+c(K_2)$. Then the inequality Equation \ref{eq:ineq} is obviously an equality, which completes the proof of the proposition.\end{proof}

\section{Higher degree theta-curves}
\label{sec:higher}
In the previous section theta-curves are shown to be closely related to composite knots. A next plausible step is to add more arcs between the two nodes.
In this section we consider graphs that have two nodes and $2n$ arcs between them, i.e. $2n$-theta-curves or theta-curves of degree $2n$. We sometimes refer to theta-curves with 3 edges and 2 vertices as \textit{classical theta-curves} or \textit{theta-curves of degree 3}.

Again there is a unique planar embedding of this graph, the trivial theta-curve of degree $2n$ as in Figure \ref{fig:thetan}a). Tying knots into the different arcs is still a well-defined operation and we can thus study the minimal crossing number of the graph $\theta_{K_{1},K_{2}}^{n}$ which is obtained from the trivial theta-curve of order $2n$ by tying $K_{1}$ into $n$ arcs and $K_{2}$ into the remaining $n$ arcs (cf. Figure \ref{fig:thetan}b)). Note that $\theta_{K_1,K_2}^1$ is simply the connected sum $K_{1}\# K_{2}$. 

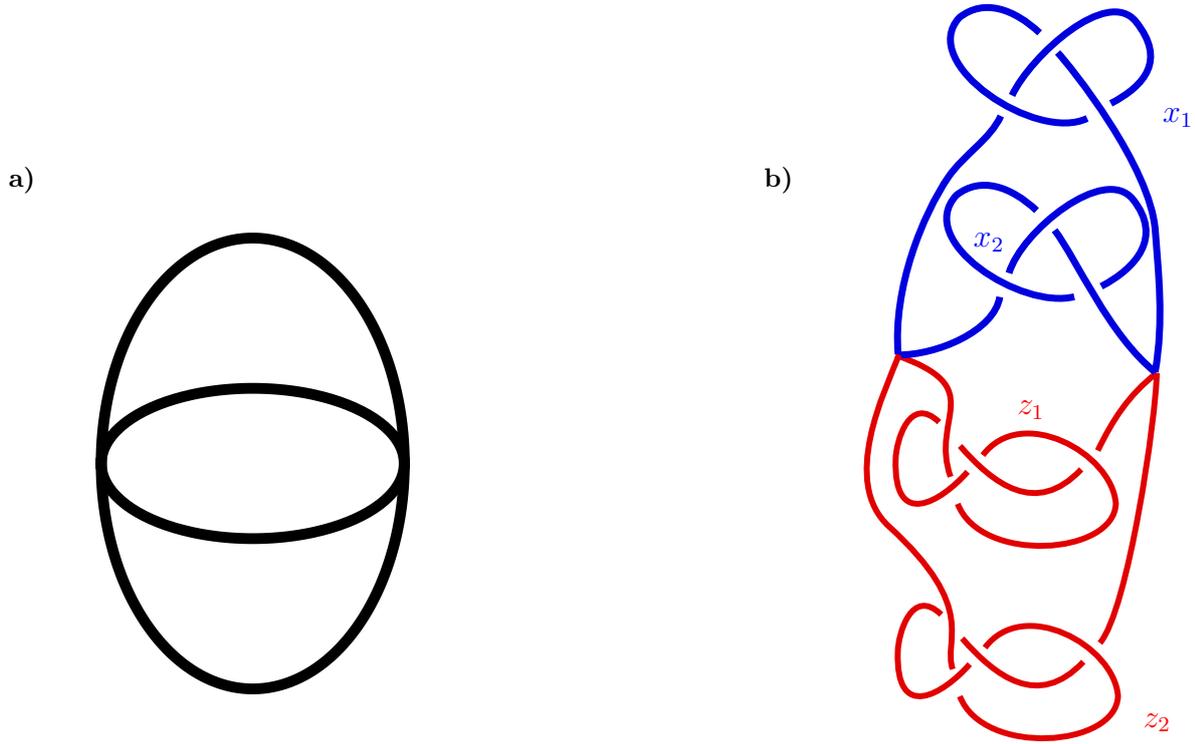
\begin{figure}[tb]\centering
\labellist
\pinlabel \textbf{a)} at 20 280
\pinlabel \textbf{b)} at 380 280
\Large
\pinlabel ${\color{blue}x_1}$ at 570 310
\pinlabel ${\color{blue}x_2}$ at 480 250
\pinlabel ${\color{red}z_1}$ at 500 170
\pinlabel ${\color{red}z_2}$ at 560 20
\endlabellist
\begin{subfigure}[t]{0.45\textwidth}
\includegraphics[width=\linewidth]{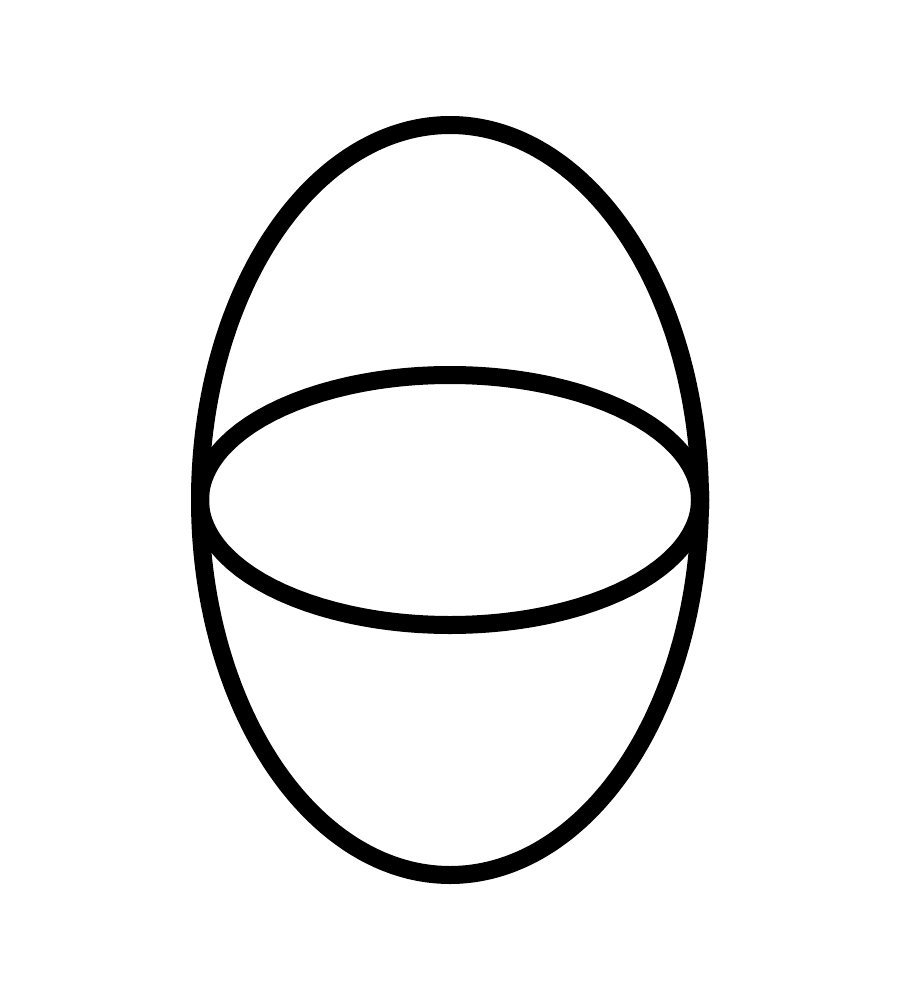}
\end{subfigure}\hfill
\begin{subfigure}[t]{0.3\textwidth}
\includegraphics[width=\linewidth, valign=b]{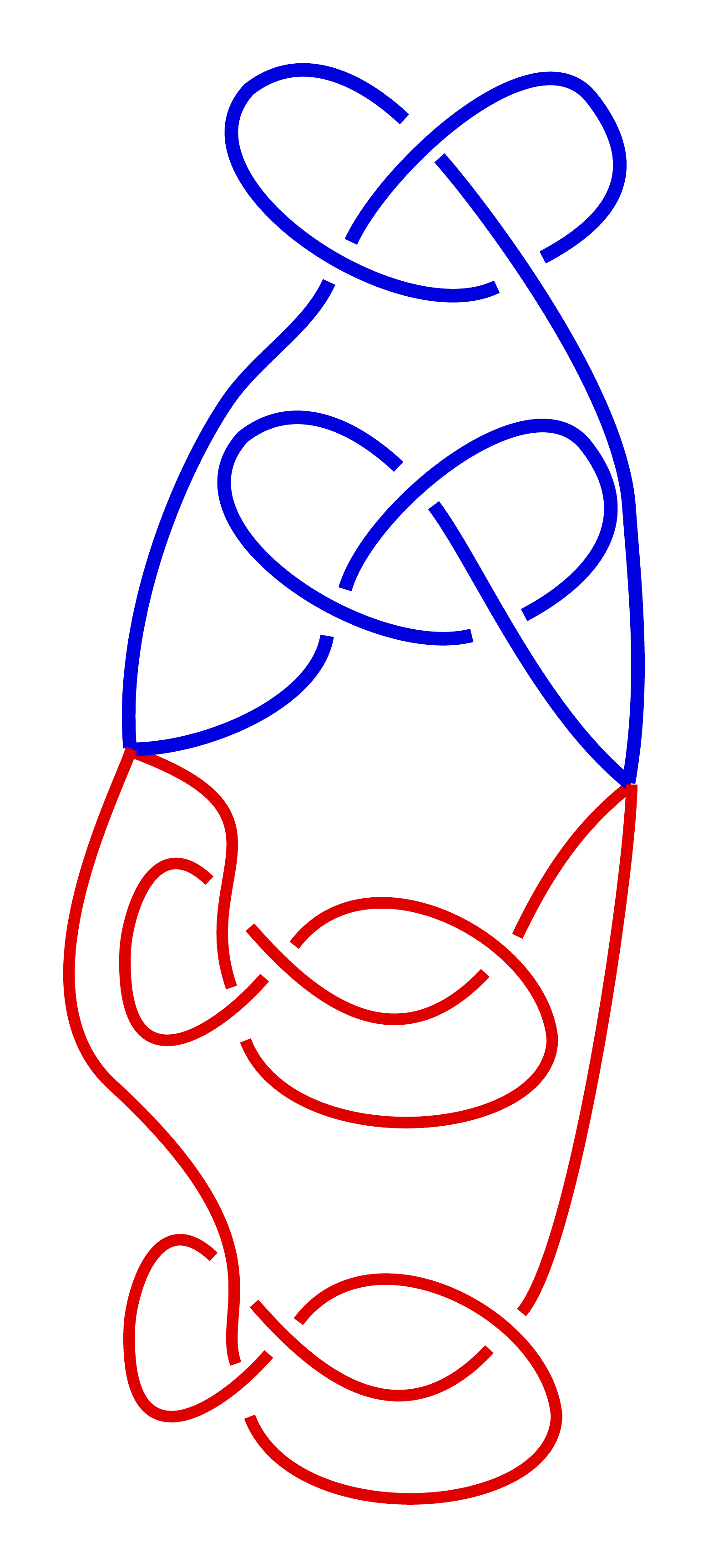}
\end{subfigure}
\caption{a) The planar embedding of the 4-theta-graph. b) A diagram for the ambient isotopy type
$\theta_{K_1=3_1,K_2=4_1}^2$.\label{fig:thetan}}
\end{figure}

We label the edges with a $K_1$ in it by $x_1, \ldots,x_n$ and the edges with a $K_2$ in it by $z_1, \ldots,z_n$. We thus obtain the following constituent knots: $x_i\cup z_j=K_1\# K_2$, $x_i\cup x_j=K_1\# K_1$ and $z_i\cup z_j=K_2\# K_2$ for all distinct $i,j\in\{1,\ldots,n\}$.

We adopt the notation from the previous sections, so $x_i x_j$ denotes the number of times the edge $x_i$ crosses the edge $x_j$. Analogous notations hold for the other edges.  

The first thing that we should note is a direct corollary from Lemma \ref{jordan}.

\begin{corollary}
\label{mixed}
For all knots $K_{1}$, $K_2$ and all $n\in\mathbb{N}$ we have that $c(\theta_{K_1,K_2}^n)\geq nc(K_1\# K_2)$. There is a $n>1$ for which equality holds if and only $c(K_1\# K_2)=c(K_1)+c(K_2)$. 
\end{corollary}

\begin{proof}
The inequality follows directly from the definition of $\theta_{K_1,K_2}$, in particular from the fact that $x_i\cup z_j=K_1\# K_2$ for all $i$ and $j$.
In other words, for all $k\in\{0,1,\ldots,n-1\}$ we have 
\begin{align}
c(\theta_{K_1,K_2}^n)\geq &\sum_{i=1}^n (x_i x_i+x_i z_{1+(i+k)\text{ mod }n}+z_{1+(i+k)\text{ mod }n} z_{1+(i+k) \text{ mod }n})\nonumber\\
& + \underset{i> j}{\sum_{i,j=1}^n}(x_i x_j+z_i z_j)+\underset{j\neq 1+(i+k)\text{ mod }n)}{\sum_{i,j=1}^n}x_i z_j. 
\end{align} 

Summing over all $k$ and using that $x_i\cup z_j=K_1\# K_2$ for all $i$, $j$, we get
\begin{equation}
nc(\theta_{K_1,K_2})\geq n^2 c(K_1\# K_2)+(n-1)\sum_{i,j=1}^n x_i z_j+n\underset{i> j}{\sum_{i,j=1}^n}(x_i x_j+z_i z_j).
\end{equation}

Thus $c(\theta_{K_1,K_2})\geq nc(K_1\# K_2)$ and if equality holds, then there are no crossings between different edges.

Hence in this case every edge is part of a theta-curve (as in Section \ref{sec:proof}), where none of the strands cross each other. It follows from Lemma \ref{jordan} that each edge crosses itself at least $c(K_{i})$, $i=1,2$ number of times, respectively, meaning $x_i x_i\geq c(K_1)$ and $z_i z_i\geq c(K_2)$ for all $i\in\{1,\ldots,n\}$. Thus $c(\theta_{K_1,K_2}^n)=n(c(K_1)+c(K_2))$ and since $c(\theta_{K_1,K_2}^n)=nc(K_1\# K_2)$ by assumption, we have $c(K_1\# K_2)=c(K_1)+c(K_2)$. 

If $c(K_1\# K_2)=c(K_1)+c(K_2)$, then $c(\theta_{K_1,K_2}^n)\geq n(c(K_1)+c(K_2))$ for all $n\in\mathbb{N}$. Since on the other hand $c(\theta_{K_1,K_2}^n)\leq n(c(K_1)+c(K_2))$ for all $n\in\mathbb{N}$, we obtain $c(\theta_{K_1,K_2}^n)=n(c(K_1)+c(K_2))=nc(K_1\# K_2)$ for all $n\in\mathbb{N}$, which proves the corollary.
\end{proof}

We can also relate the crossing numbers of $\theta_{K_1,K_1}^{n}$ and the connected sum of $n$ copies of $K_1\# K_2$, denoted by $K_1^n\# K_2^n$.

\begin{proposition}
\label{resolve}For all knots $K_1$ and $K_2$ and all $n\in\mathbb{N}$ we have $c(\theta_{K_1,K_2}^n)\geq c(K_1^n\# K_2^n)$. There is one $n$ for which equality holds if and only if $c(K_1\# K_2)=c(K_1)+c(K_2)$.
\end{proposition}
\begin{proof}
The key idea here is that we can take any diagram of $\theta_{K_1,K_2}^n$ and resolve the two nodes in a certain way (as in Figure \ref{fig:resolve}) such that we obtain a diagram of $K_1^n\# K_2^n$. We do this as follows. We start at one of the nodes, say $n_{1}$ and pick any arc $s_{1}$. We follow it along the diagram until it reaches the other node $n_{2}$. We then have to pick another arc $s_{2}$ to connect with $s_{1}$. We define $s_{2}$ to be the arc which enters $n_2$ next to $s_1$ in the clockwise direction.

We then follow $s_{2}$ along the diagram until it reaches $n_{1}$ and pick $s_{3}$ to be the arc which among all strands that we have not picked yet enters $n_{1}$ the closest to $s_{2}$ in the clockwise direction. In general, we connect the arc $s_{i}$ to the arc $s_{i+1}$, where $s_{i+1}$ is the arc that among all arcs that are not an element of $\{s_{1},s_{2},\ldots,s_{i}\}$ enters the node $n_{(i\text{ mod }2)+1}$ closest to $s_{i}$ in the clockwise direction.

With this rule, we obtain only one connected component, i.e. the diagram of a knot. It is clear, for example through induction on $n$, that the knot type of this diagram is $K_1^n\# K_2^n$.

Assume now that there is an $n$ such that $c(\theta_{K_1,K_2}^n)=c(K_1^n\# K_2^n)$. Note that we have $c(K_1^n\# K_2^n)\leq n c(K_1\# K_2)\leq c(\theta_{K_1,K_2}^n)$. It then follows from Corollary \ref{mixed} that $c(\theta_{K_1,K_2}^n)=c(K_1^n\# K_2^n)=n c(K_1\# K_2)$ implies $c(K_1\# K_2)=c(K_1\# K_2)$.

If $c(K_1\# K_2)=c(K_1)+c(K_2)$, then Equation \ref{eq:1} implies that $c(\theta_{K_1,K_2})\geq c(K_1)+c(K_2)$. However, we know from the definition of $\theta_{K_1,K_2}$ that $c(\theta_{K_1,K_2})\leq c(K_1)+c(K_2)$ and therefore $c(\theta_{K_1,K_2})=c(K_1)+c(K_2)=c(K_1\# K_2)$. Since $\theta_{K_1,K_2}=\theta_{K_1,K_2}^1$, this proves the proposition. 
\end{proof}

\begin{figure}[tb]\centering
\labellist
\pinlabel \textbf{a)} at 2 250
\pinlabel \textbf{b)} at 430 250
\endlabellist
\begin{subfigure}[t]{0.45\textwidth}
\includegraphics[width=0.45\linewidth]{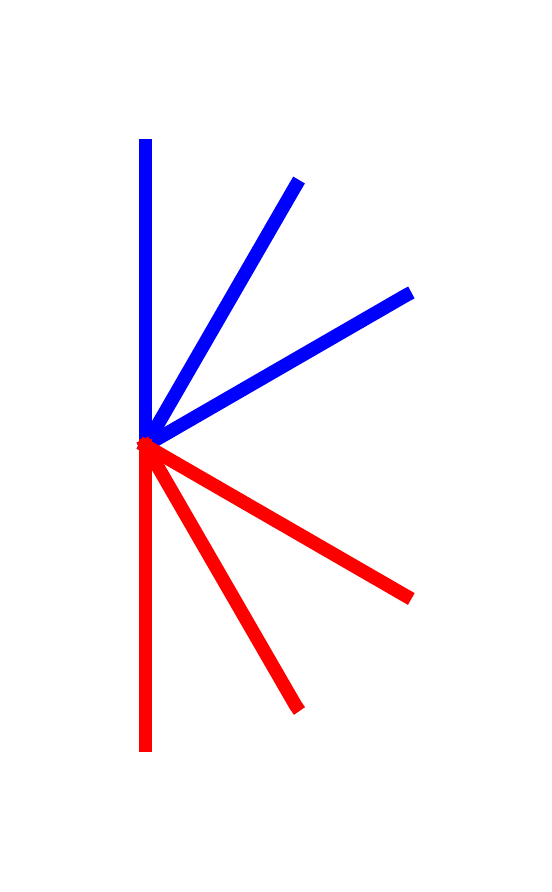}
\includegraphics[width=0.45\linewidth, valign=b]{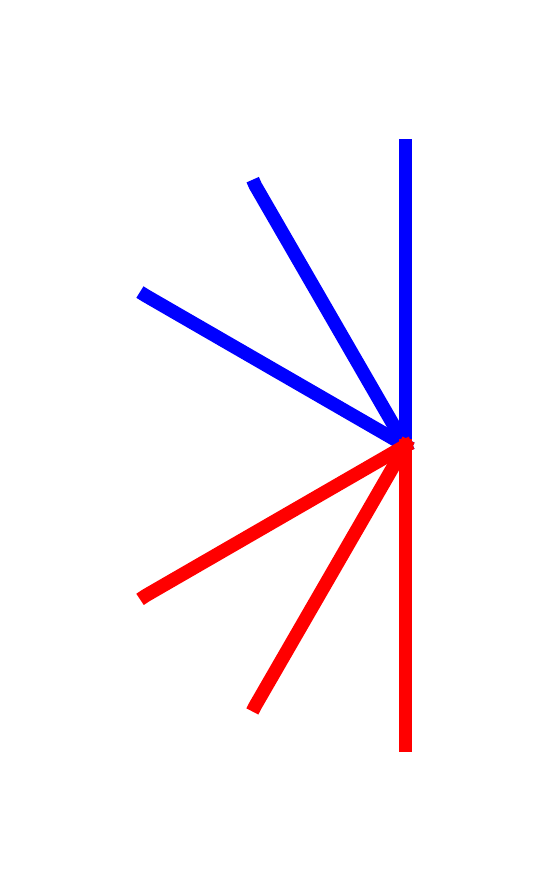}
\end{subfigure}\hfill
\begin{subfigure}[t]{0.45\textwidth}
\includegraphics[width=0.45\linewidth, valign=b]{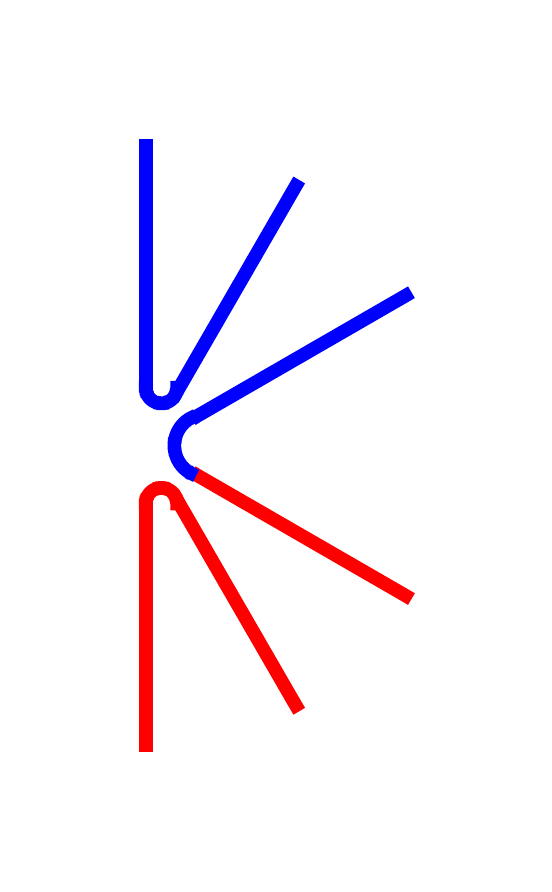}
\includegraphics[width=0.45\linewidth, valign=b]{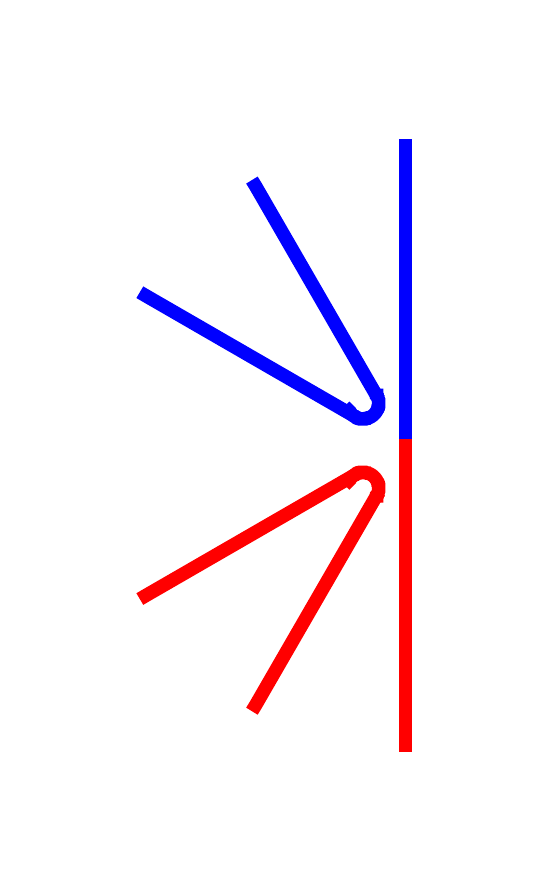}
\end{subfigure}
\caption{a) Neighbourhoods of the nodes in a diagram of $\theta_{K_1,K_2}^n$. b) The nodes can be resolved to result in a diagram of $K_1^n\# K_2^n$.\label{fig:resolve}}
\end{figure}

The graph $\theta_{K_{1},K_{2}}^{n}$ is an element of a special class of theta-curves of degree $2n$. We define $\Omega_{K_{1},K_{2}}^{n}$ to be the set of theta-curves of degree $2n$ where we can colour $n$ arcs blue and the remaining $n$ arcs red, such that the union of any blue arc with any red arc is $K_{1}\# K_{2}$ and the union of any two arcs of the same colour is neither the unknot nor $K_1\# K_2\# K_1\# K_2$. Obviously $\theta_{K_{1},K_{2}}\in\Omega_{K_{1},K_{2}}^{n}$.

In order to keep notation consistent with that of the discussion of $\theta_{K_1,K_2}$, we label the blue edges by $x_1,\ldots,x_n$ and the $n$ red edges by $z_1,\ldots,z_n$.

We are now interested in $c(\Omega_{K_{1},K_{2}}^{n})=\min \{c(\theta):\theta\in\Omega_{K_{1},K_{2}}^{n}\}$. By the above we have $c(\Omega_{K_{1},K_{2}}^{n})\leq c(\theta_{K_{1},K_{2}}^{n})\leq n(c(K_{1})+c(K_{2}))$. 
We want to show that for large enough $n$ this inequality is actually an equality. The idea here is that any three arcs of a theta-curve of order $2n$ form a `classical' theta-curve as in the previous section and we either have an intersection between a pair of arcs or the crossing number of the theta-curve is in some sense large. However, as $n$ grows, the number of pairs of arcs grows more quickly than $n(c(K_{1})+c(K_{2}))$.

We need several lemmas.

\begin{lemma}
\label{decomposition}
Let $\theta$ be a theta-curve with $x\cup z=K_{1}\# K_{2}$ and $y\cup z=K_{1}\# K_{2}$. If no pair of arcs cross each other and $x\cup y$ is neither the unknot nor $K_1\# K_2\# K_1\# K_2$, then $xx\geq c(K_{1})$, $zz\geq c(K_{2})$ and $yy\geq c(K_{1})$ or $xx\geq c(K_{2})$, $zz\geq c(K_{1})$ and $yy\geq c(K_{2})$.
\end{lemma}

\begin{proof}
By Lemma \ref{jordan} there are knots $K_{1}'$, $K_{2}'$ and $K_{3}'$ such that $K_{1}'\# K_{3}'=K_{2}'\# K_{3}'=K_{1}\# K_{2}$ and $K_{1}'\# K_{2}'$ is neither the unknot nor $K_1\# K_1\# K_2\# K_2$.

Since the prime decomposition of knots is unique and both $K_1$ and $K_2$ are prime, $K_{1}'$ is either $K_{1}$, $K_{2}$, $K_{1}\# K_{2}$ or the unknot.
If it is the unknot, then $K_{3}'=K_{1}\# K_{2}$. But then $K_{2}'$ must also be the unknot and so $K_{1}'\# K_{2}'$ is the unknot, contradicting the assumption.

If $K_{1}'=K_{1}\# K_{2}$, then $K_{3}'$ is the unknot and hence $K_{2}'=K_1\# K_2$. Thus $K_{1}'\# K_{2}'=K_1\# K_2 \# K_1\# K_2$, again contradicting the assumption. 

If $K_{1}'=K_{1}$, then $K_{3}'=K_{2}$ and therefore $K_{2}'=K_{1}$ and so $xx\geq c(K_{1})$, $yy\geq c(K_{1})$ and $zz\geq c(K_{2})$ by Lemma \ref{jordan}.

If $K_{1}'=K_{2}$, then $K_{3}'=K_{1}$ and hence $K_{2}'=K_{2}$. It follows that $xx\geq c(K_{2})$, $yy\geq c(K_{2})$ and $zz\geq c(K_{1})$ by Lemma \ref{jordan}.
\end{proof}

This establishes the idea that if a theta-curve of degree 3 that is a subgraph of the diagram in question consists of three arcs that do not cross each other (only themselves), then its crossing numbers is comparatively large. We are thus interested in how many crossings between different edges are required to rule out the existence of any such subgraph.

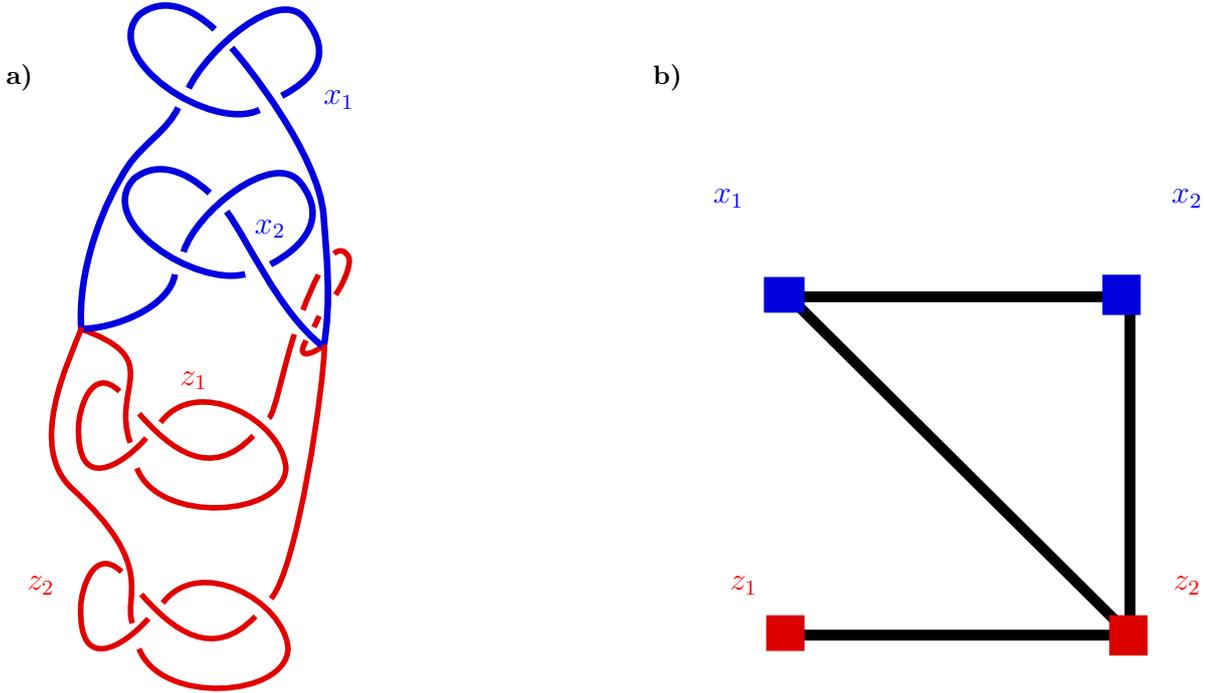
\begin{figure}[tb]\centering
\labellist
\pinlabel \textbf{a)} at 2 850
\pinlabel \textbf{b)} at 850 850
\Large
\pinlabel ${\color{blue}x_1}$ at 420 820
\pinlabel ${\color{blue}x_2}$ at 330 650
\pinlabel ${\color{red}z_1}$ at 230 450
\pinlabel ${\color{red}z_2}$ at 30 180
\pinlabel ${\color{blue}x_1}$ at 930 690
\pinlabel ${\color{blue}x_2}$ at 1530 690
\pinlabel ${\color{red}z_1}$ at 950 180
\pinlabel ${\color{red}z_2}$ at 1530 180
\endlabellist
\begin{subfigure}[t]{0.3\textwidth}
\includegraphics[width=\linewidth]{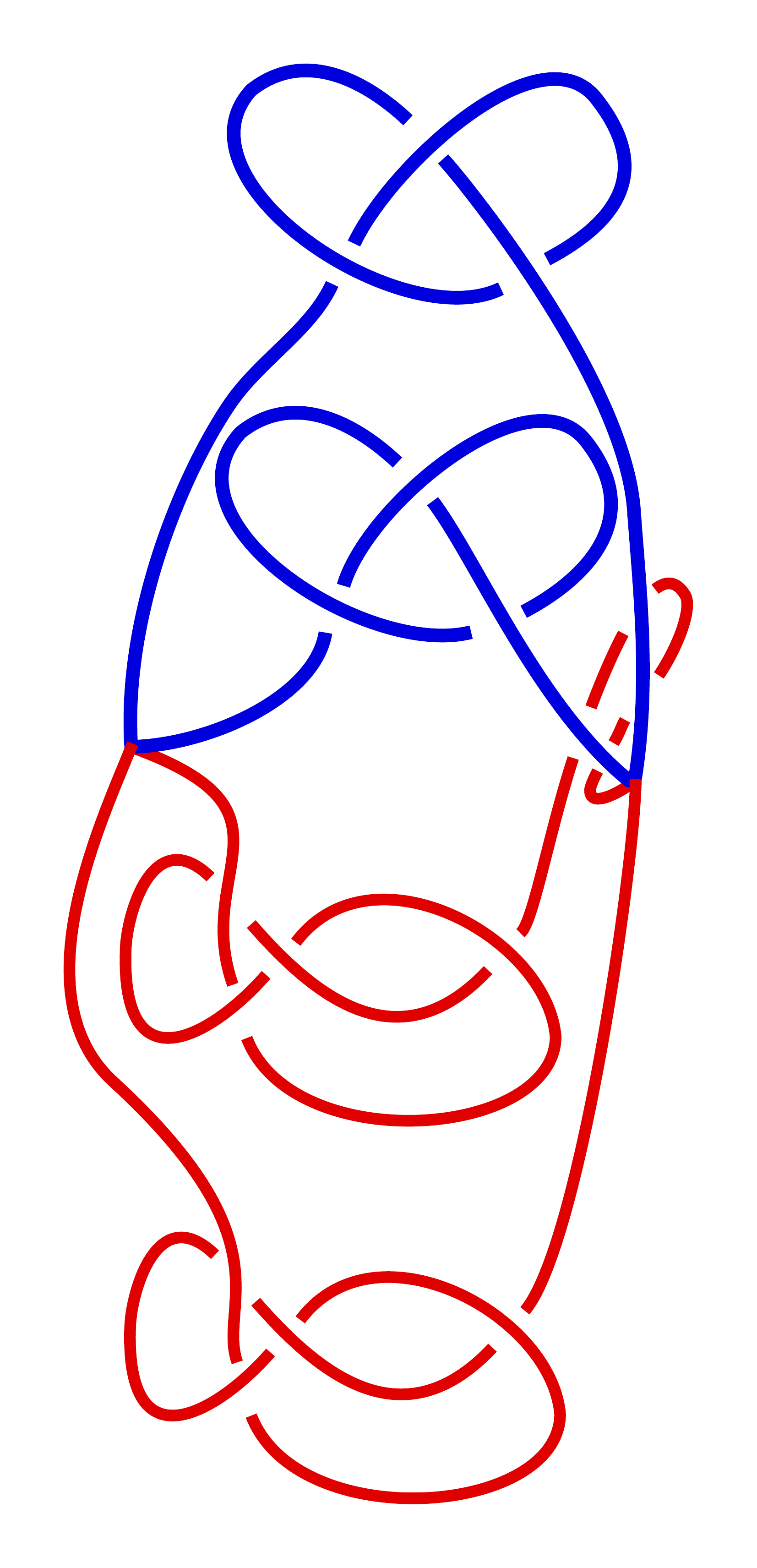}
\end{subfigure}\hfill
\begin{subfigure}[t]{0.45\textwidth}
\vspace{-7cm}
\includegraphics[width=\linewidth]{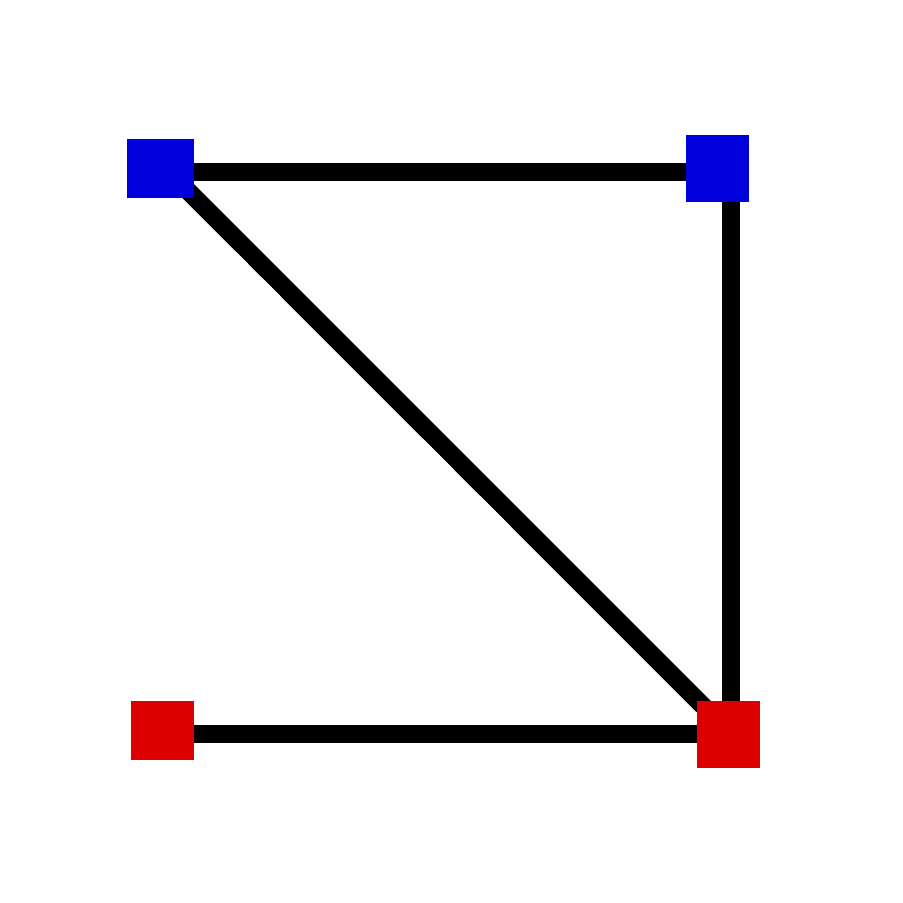}
\end{subfigure}
\caption{a) A diagram of $\theta_{K_1=3_1,K_2=4_1}^2$. b) The corresponding $\Gamma$-graph. The vertices $x_1$, $x_2$ and $z_2$ form a bicoloured triangle.\label{fig:gamma}}
\end{figure}

We can associate a graph $\Gamma$, or $\Gamma(D)$, to any diagram $D$ of a theta-curve $\theta\in\Omega_{K_1,K_2}^n$ that consists of $2n$ vertices, one for each edge of $D$, and an edge between two vertices if the corresponding edges in $D$ do not cross each other. Hence there is an edge between the vertices corresponding to $x_i$ and $z_j$ if and only if $x_i z_j=0$. Similarly, for $x_i$ and $x_j$ or $z_i$ and $z_j$.

We call a triangle in $\Gamma$ \textit{bicoloured} if its set of vertices consists of $x$'s and $z$'s, i.e. either $(x_i,x_j,z_k)$ or $(x_i,z_j,z_k)$.

Note that three arcs $(x_i,x_j,z_k)$ or $(x_i,z_j,z_k)$ in the diagram $D$ form a theta-curve as in Lemma \ref{decomposition} if and only if their corresponding vertices in $\Gamma(D)$ form a bicoloured triangle.

\begin{lemma}
\label{graph}
Let $n\geq 2$ and $\Gamma$ be a graph with $2n$ vertices, labelled $x_1,\ldots,x_n$, $z_1,\ldots,z_n$, and $m$ edges. If 
\begin{equation}
m\geq \frac{3}{2}n^2-n,
\end{equation}
then $\Gamma$ contains a bicoloured triangle.
\end{lemma}

\begin{proof}
Let $d(v)$ denote the degree of the vertex $v$. Let $V(\Gamma)$ and $E(\Gamma)$ denote the set of vertices and edges of $\Gamma$ respectively. Note that 
\begin{align}
\sum_{x\in V(\Gamma)}d^2 (x)&=\sum_{(x,y)\in E(\Gamma)}(d(x)+d(y))\nonumber\\
&=\sum_{(x_i,x_j)\in E(\Gamma)}(d(x_i)+d(x_j))+\sum_{(x_i,z_j)\in E(\Gamma)}(d(x_i)+d(z_j))\nonumber\\
&+\sum_{(z_i,z_j)\in E(\Gamma)}(d(z_i)+d(z_j)).
\end{align}
Assume now that $\Gamma$ does not contain a bicoloured triangle. Then if there is an edge between $x_i$ and $x_j$ every $z_k$ is directly connected to at most one of them. Thus $d(x_i)+d(x_j)\leq 2(n-1)+n=3n-2$. Similarly, $d(z_i)+d(z_j)\leq 3n-2$, whenever there is an edge between $z_i$ and $z_j$.

If there is an edge between $x_i$ and $z_j$ every other vertex is directly connected to at most one of $x_i$ and $z_j$. Thus $d(x_i)+d(z_j)\leq 2n$. We obtain if $n\geq 2$

\begin{equation}
\sum_{x\in V(\Gamma)}d^2 (x)\leq m(3n-2).
\end{equation}

Furthermore, since $\sum_{x\in V(\Gamma)}d(x)=2m$, the Cauchy-Schwartz inequality implies that
\begin{equation}
\sum_{x\in V(\Gamma)}d^2 (x)\geq \frac{\left(\sum_{x\in V(\Gamma)}d(x)\right)^2}{2n}=\frac{4m^2}{2n}.
\end{equation}

Thus $\frac{2m^2}{n}\leq m(3n-2)$ and we obtain $m\leq \frac{3}{2}n^2-n$.
\end{proof}

\begin{lemma}
\label{largek}
If $n> 2(c(K_1)+c(K_2)-c(K_1\# K_2))+1$, then for every diagram $D$ of $\theta\in\Omega_{K_{1},K_{2}}^{n}$ with $c(D)\leq n(c(K_{1})+c(K_{2}))$ there is a bicoloured triangle in $\Gamma(D)$.
\end{lemma}

\begin{proof}
Since $x_i\cup z_j=K_{1}\# K_{2}$ for all $i,j$, we have the inequality
\begin{equation}
\label{eq:kineq}
n c(D)\geq n^2 c(K_{1}\# K_{2})+(n-1)\sum_{i,j=1}^{n}x_{i}z_{j}+n\underset{i> j}{\sum_{i,j=1}^n}(x_i x_j+z_i z_j).
\end{equation}

Assume there is no bicoloured triangle in $\Gamma(D)$. Then by Lemma \ref{graph} $\Gamma(D)$ has at most $\tfrac{3}{2} n^2-n$ edges. Thus for at most $\tfrac{3}{2} n^2-n$ pairs of arcs there is no crossing between them. Hence for at least $\tfrac{2n(2n-1)}{2}-\tfrac{3}{2} n^2+n=\tfrac{n^2}{2}$ pairs there is a crossing between them. Note that since we only count crossings of $x_{i}$ with $x_{j}$ and $z_{j}$ and crossings of $z_{i}$ with $x_{j}$ and $z_{j}$, we count every crossing only once.

Equation (\ref{eq:kineq}) then becomes
\begin{equation}
n c(D)\geq n^2 c(K_{1}\# K_{2})+\frac{(n-1)n^2}{2}.
\end{equation}

With the assumption that $c(D)\leq n(c(K_{1})+c(K_{2}))$ we get 
\begin{align}
n^{2} (c(K_{1})+c(K_{2}))\geq n^2 c(K_{1}\# K_{2})+\frac{(n-1)n^2}{2}\nonumber\\
\implies(c(K_{1})+c(K_{2})-c(K_{1}\# K_{2}))\geq \frac{n-1}{2},
\end{align}
which gives a contradiction if $n> 2(c(K_{1})+c(K_{2})-c(K_{1}\# K_{2}))+1$.

Thus $\Gamma(D)$ does contain a bicoloured triangle $(x_i,x_j,z_k)$ or $(x_i,z_j,z_k)$ if $n>2(c(K_{1})+c(K_{2})-c(K_{1}\# K_{2}))+1$. 
\end{proof}

Note that Lemma \ref{largek} directly implies the following result.
\begin{lemma}
\label{lowl} 
For $n=2(c(K_{1})+c(K_{2})-c(K_{1}\# K_{2}))+1+k$, $k\geq 1$, a diagram $D$ of $\theta\in\Omega_{K_{1},K_{2}}^{n}$ with $c(D)\leq n(c(K_{1})+c(K_{2}))$ has at least $k$ arcs $x_{i}$ with $x_{i}x_{i}\geq c(K_{1})$ and at least $k$ arcs $z_{i}$ with $z_{i}z_{i}\geq c(K_{2})$. Each of these $x_{i}$ and $z_{i}$ is part of a classical theta-curve where no pair of arcs is crossing each other.
\end{lemma}

\begin{proof} 
By Lemma \ref{largek} $\Gamma(D)$ contains a bicoloured triangle. By Lemma \ref{decomposition} this means that the arcs $(x_i,x_j,z_t)$ or $(x_i,z_j,z_t)$ of $D$ that correspond to the vertices of the bicoloured triangle satisfy $x_ix_i\geq c(K_1)$, $z_tz_t\geq c(K_2)$ and $x_jx_j\geq c(K_1)$ or $z_jz_j\geq c(K_2)$.

Deleting $x_{i}$ and $z_t$ results in a diagram $D'$ of a theta-curve of degree $2n-2$ in $\Omega_{K_{1},K_{2}}^{n-1}$ with $c(D')\leq (n-1)(c(K_{1})+c(K_{2}))$. Repeatedly applying Lemma \ref{largek} and Lemma \ref{decomposition} results in the proof of the lemma.  
\end{proof}

\begin{proposition}
If $n\geq\max\{4(c(K_{1})+c(K_{2})-c(K_1\# K_2))+2,2(c(K_1)+c(K_2)+1)\}$, then $c(\Omega_{K_{1},K_{2}}^{n})=n(c(K_{1})+c(K_{2}))$.
\end{proposition}
\begin{proof}
Assume $c(\Omega_{K_{1},K_{2}}^{n})<n(c(K_{1})+c(K_{2}))$ and let $D$ be a diagram of a theta-curve of degree $2n$ that is in $\Omega_{K_{1},K_{2}}^{n}$ such that $c(D)<n(c(K_{1})+c(K_{2}))$. 

Let $l$ be the largest integer such that there are $l$ arcs $x_i$ and $l$ arcs $z_i$ with $x_i x_i\geq c(K_1)$ and $z_i z_i\geq c(K_2)$ whose corresponding vertices in $\Gamma(D)$ are not part of a bicoloured triangle.
We label these arcs by $x_i$ and $z_i$, $i=1,\ldots,l$.

Let $k$ be the largest integer such that there are $k$ arcs $x_{i}$ and $k$ arcs $z_{i}$ whose corresponding vertices in $\Gamma(D)$ are part of a bicoloured triangle in $\Gamma(D)$. Then by Lemma \ref{jordan} these arcs each cross themselves at least $c(K_{1})$ and $c(K_{2})$ times, respectively, i.e. $x_ix_i\geq c(K_1)$ and $z_iz_i\geq c(K_2)$.
We label these arcs $x_{i}$ and $z_{i}$, $i=l+1,l+2,\ldots,l+k$.

Let $\tilde{D}$ denote the diagram that results from deleting the arcs $x_i$ and $z_i$, $i=1,\ldots,l$. Note that $c(D)\geq c(\tilde{D})+l(c(K_1)+c(K_2))$.

We therefore have
\begin{align}
(n-l) c(\tilde{D})= &(n-l)\left(\underset{i\geq j}{\sum_{i,j=l+1}^{l+k}}(x_ix_j+z_iz_j)+\sum_{i,j=l+1}^{l+k}x_iz_j\right.\nonumber\\
&\left.+\sum_{j=l+1}^{l+k}\sum_{i=l+k+1}^n (x_ix_j+x_iz_j+z_ix_j+z_iz_j)\right.\nonumber\\
&\left.+\underset{i\geq j}{\sum_{i,j=l+k+1}^n} (x_ix_j+z_iz_j)+\sum_{i,j=l+k+1}^n x_iz_j\right).
\end{align}

Rearranging the terms on the right hand side gives
\begin{align}
\label{eq:kquadrat}
&\sum_{i,j=l+1}^{l+k}(x_ix_i+x_iz_j+z_jz_j)+\sum_{i=l+k+1}^{n}\sum_{j=l+1}^{l+k}(x_ix_i+x_iz_j+z_jz_j)\nonumber\\
&+\sum_{i=l+k+1}^{n}\sum_{j=l+1}^{l+k}(x_jx_j+x_jz_i+z_iz_i)
\sum_{i,j=l+k+1}^n (x_ix_i+x_iz_j+z_jz_j)\nonumber\\
&+(n-l-1)\sum_{i,j=l+1}^n x_iz_j+(n-l)\underset{i>j}{\sum_{i,j=l+1}^n}(x_ix_j+z_iz_j).
\end{align}

Since $x_ix_i+x_iz_j+z_jz_j\geq c(K_1\# K_2)$ for all $i,j\in\{1,2,\ldots,n\}$, and $x_ix_i\geq c(K_1)$ and $z_iz_i\geq c(K_2)$ for all $i\in\{l+1,l+2,\ldots,l+k\}$, Equation (\ref{eq:kquadrat}) is at least 
\begin{align}
&k^2(c(K_1)+c(K_2))+(n-l)k(c(K_{1})+c(K_{2}))+(n-l-k)^2 c(K_{1}\# K_{2})\nonumber\\
&+(n-l-1)\sum_{i,j=l+1}^{n}x_{i}z_{j}+(n-l)\underset{i>j}{\sum_{i,j=l+1}^{n}}(x_{i}x_{j}+z_{i}z_{j}).
\end{align}

It follows from $k\geq 0$ and $c(D)< n(c(K_{1})+c(K_{2}))$ and therefore $c(\tilde{D})<(n-l)(c(K_1)+c(K_2))$ that
\begin{align}
(n-l)^2 (c(K_{1})+c(K_{2}))&>(n-l)k(c(K_{1})+c(K_{2}))\nonumber\\
&+(n-l-k)^2 c(K_{1}\# K_{2})\nonumber\\
&+(n-l-1)\sum_{i,j=l+1}^{n}x_{i}z_{j}\nonumber\\
&+(n-l)\underset{i>j}{\sum_{i,j=l+1}^{n}}(x_{i}x_{j}+z_{i}z_{j})\nonumber\\
\iff (n-l-k)(n-l)(c(K_{1})+c(K_{2}))&>(n-l-k)^2 c(K_{1}\# K_{2})\nonumber\\
&+(n-l-1)\sum_{i,j=l+1}^{n}x_{i}z_{j}\nonumber\\
&+(n-l)\underset{i>j}{\sum_{i,j=l+1}^{n}}(x_{i}x_{j}+z_{i}z_{j}).
\end{align}

By construction there can not be any bicoloured triangles in the $\Gamma$-graph associated to the theta-curve of order $2(n-l-k)$ that results from $\tilde{D}$ by deleting the $x_{i}$ and $z_{i}$ with $i=l+1,l+2,\ldots,l+k$. Thus there are at least $\frac{1}{2}(n-l-k)^2$ crossings between arcs with indices larger than $l+k$.

Furthermore, by definition of $l$ and $k$ for every $i>l+k$ either $x_{i}$ or $z_{i}$ must cross $x_j$ or $z_j$ for all $j=l+1,l+2,\ldots,l+k$ at least once.

This gives
\begin{align}
(n-l-k)(n-l)(c(K_{1})+c(K_{2}))>&(n-l-k)^2 c(K_{1}\# K_{2})+(n-l-1)\nonumber\\
&\times\left(\frac{1}{2}(n-l-k)^2+k(n-l-k)\right).
\end{align}

Assume that $k<n-l$. Then we can divide by $(n-l-k)$ and obtain
\begin{align}
(n-l)(c(K_{1})+c(K_{2}))&>(n-l-k)c(K_{1}\# K_{2})+(n-l-1)\nonumber\\
&\times(\frac{1}{2}(n-l-k)+k)\nonumber\\
&=(n-l-k)c(K_{1}\# K_{2})+(n-l-1)\frac{n-l+k}{2}\nonumber\\
\iff (n-l)(c(K_1)&+c(K_2)-c(K_1\# K_2))+kc(K_1\# K_2)\nonumber\\
&> (n-l-1)\frac{n-l+k}{2}.
\end{align}

If $c(K_1\# K_2)<\frac{1}{2}(c(K_1)+c(K_2))$, then 

\begin{align} 
(n-l+k)(c(K_1)+c(K_2)-c(K_1\# K_2))&> (n-l-1)\frac{n-l+k}{2}\nonumber\\
\iff c(K_1)+c(K_2)-c(K_1\# K_2)&>\frac{n-l-1}{2},
\end{align} 
which leads to a contradiction if $n-l\geq 2(c(K_1)+c(K_2)-c(K_1\# K_2))+1$. Note that by Lemma \ref{lowl} we have $l\leq 2(c(K_1)+c(K_2)-c(K_1\# K_2))+1$. Therefore $k=n-l$ if $n\geq 4(c(K_1)+c(K_2)-c(K_1\# K_2))+2$, but this means that all arcs $x_i$ and $z_i$ whose corresponding vertices in $\Gamma(D)$ are not part of a bicloured triangle in $\Gamma(D)$ satisfy $x_i x_i\geq c(K_1)$ and $z_i z_i\geq c(K_2)$. Since the same is true for all arcs whose corresponding vertices in $\Gamma(D)$ are part of a bicoloured triangle, we have $c(D)\geq n(c(K_1)+c(K_2))$.

Similarly, if $c(K_1\# K_2)>\frac{1}{2}(c(K_1)+c(K_2))$, we obtain a contradiction if $n\geq 2(c(K_1)+c(K_2)+1)$. Thus if $n\geq 2(c(K_1)+c(K_2)+1)$, then $k=n-l$ and therefore $c(D)\geq n(c(K_1)+c(K_2))$.
\end{proof}

Since $\theta_{K_1,K_2}^n\in\Omega_{K_1,K_2}^n$, we immediately obtain the following result.

\begin{corollary}
If $n\geq\max\{4(c(K_{1})+c(K_{2})-c(K_1\# K_2))+2,2(c(K_1)+c(K_2)+1)\}$,
then $c(\theta_{K_{1},K_{2}}^{n})=n(c(K_{1})+c(K_{2}))$.
\end{corollary}

\section{Relations between composite knots and higher degree theta-curves}
\label{sec:double}

In this section we discuss relations between the $c(\Omega_{K_{1},K_{2}}^n)$ and $c(K_{1}\# K_{2})$. In particular, we show that $c(K_{1}\# K_{2})\geq\tfrac{1}{n^2 } c(\Omega_{K_{1},K_{2}}^n)$. From the previous section we know that if $n$ is sufficiently large, then $c(\Omega_{K_{1},K_{2}}^n)=n(c(K_{1})+c(K_{2}))$. Thus finding low values for $n$ for which this equality holds is a way to obtain lower bounds of the form $c(K_{1}\# K_{2})\geq \frac{1}{n} (c(K_{1})+c(K_{2}))$.

Consider a minimal diagram of $K_{1}\# K_{2}$ and draw $n-1$ parallel curves to the diagram in $\mathbb{R}^2$ that are at most $\epsilon$ away from $D$ for some small $\epsilon>0$. Obviously, we typically do not know what the minimal diagram looks like, but the procedure is well-defined. We can think of these curves as a link diagram $D_{n}$, where many of the crossings have no determined signs yet (cf. Figure \ref{fig:double}a)). We claim that we can choose the signs of these crossings and two points, where the parallel diagrams are glued together, such that we obtain a diagram of a theta-curve of degree $2n$ that is an element of $\Omega_{K_{1},K_{2}}^n$. In Figure \ref{fig:double} this can be done by choosing the signs such that the one copy of the knot diagram lies completely below the other. We can not assume that this is the case in general.
Note that the diagram constructed in this way has $n^2 c(K_{1}\# K_{2})$ crossings and thus $n^2 c(K_{1}\# K_{2})\geq c(\Omega_{K_{1},K_{2}}^n)$.

\begin{figure}[tb]\centering
\labellist
\pinlabel \textbf{a)} at 2 420
\pinlabel \textbf{b)} at 600 420
\endlabellist
\begin{subfigure}[t]{0.45\textwidth}
\includegraphics[width=\linewidth]{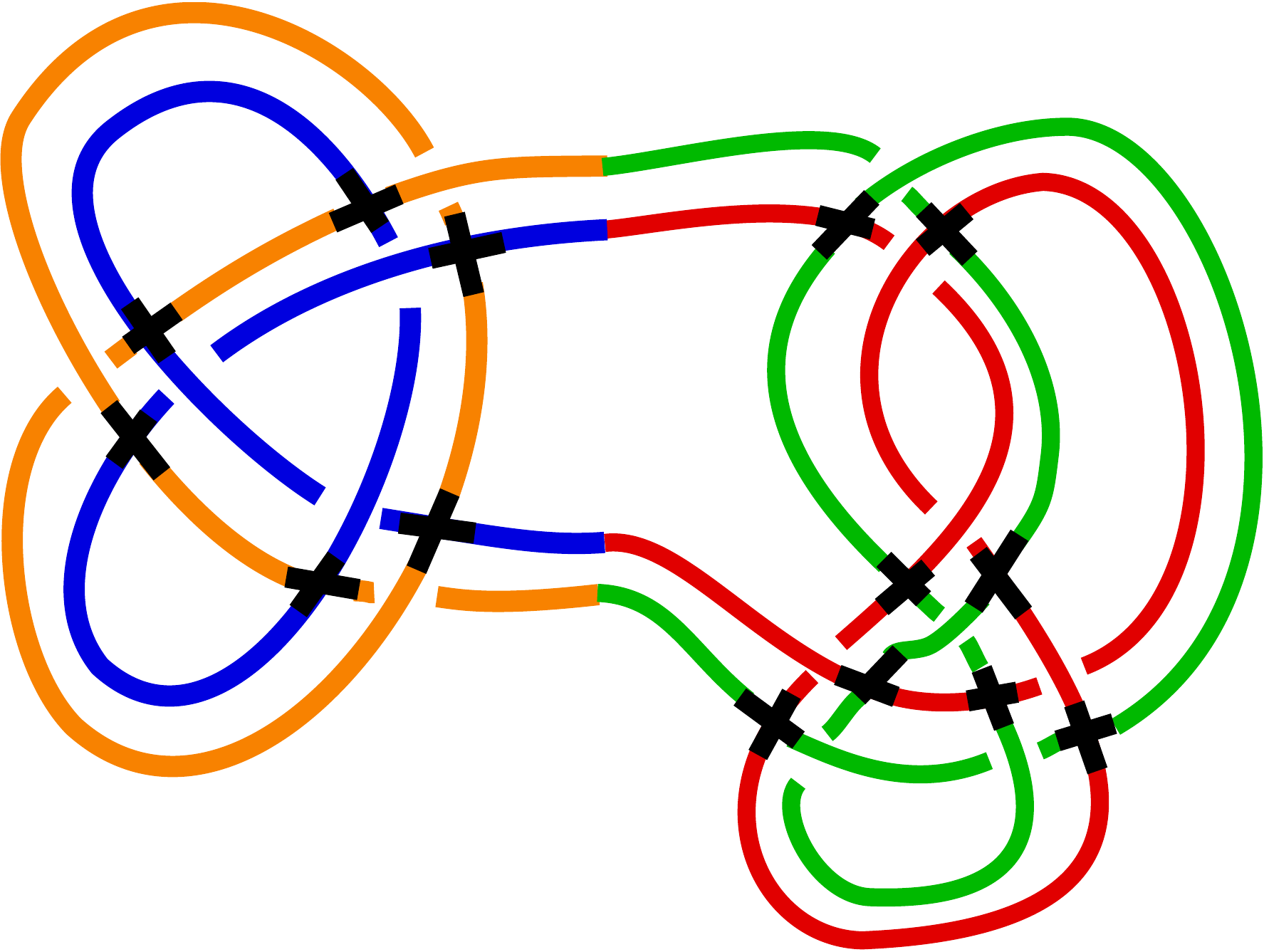}
\end{subfigure}\hfill
\begin{subfigure}[t]{0.45\textwidth}
\includegraphics[width=\linewidth, valign=b]{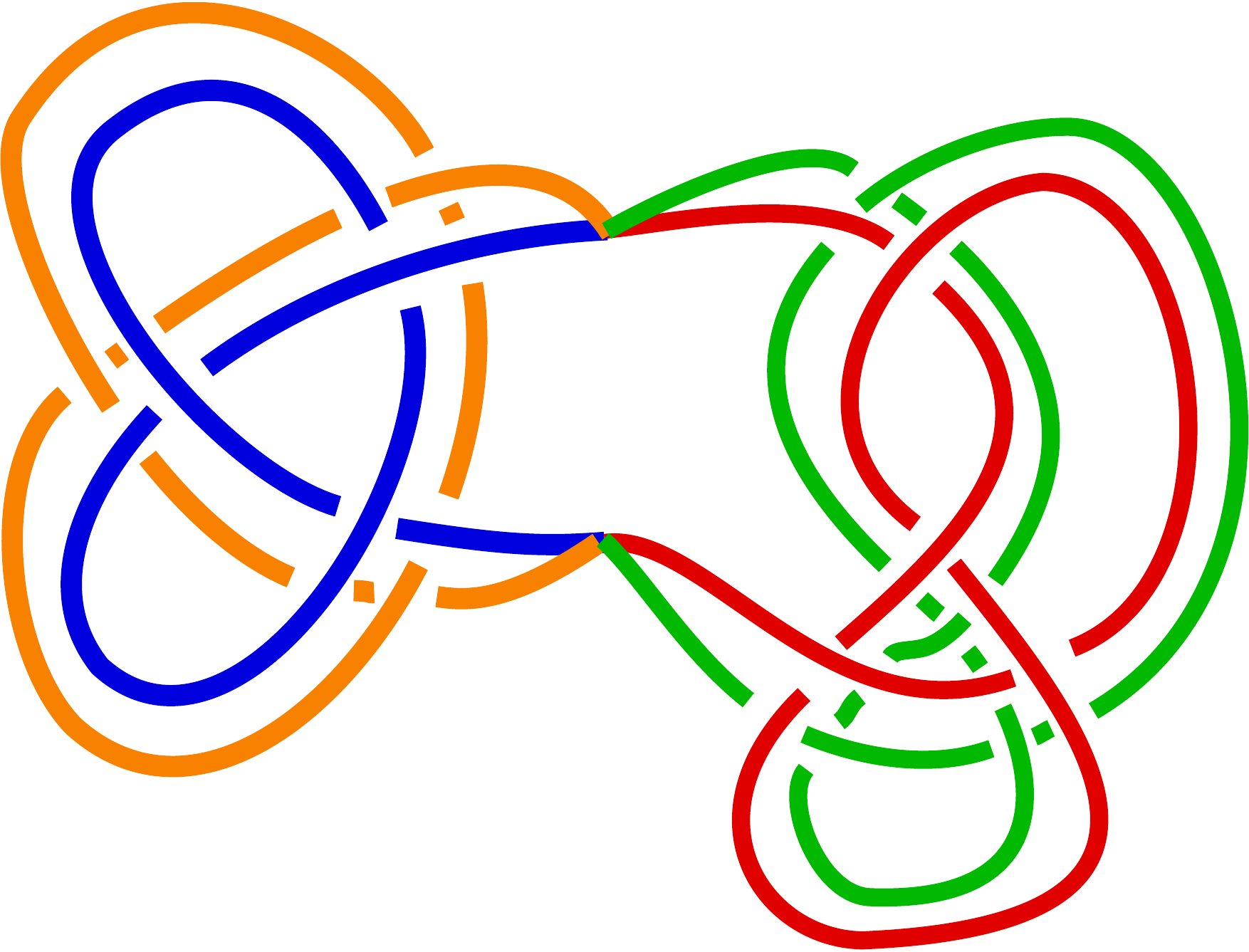}
\end{subfigure}
\caption{a) A minimal diagram of $3_1\# 4_1$ with a parallel curve next to it. The black crossings indicate crossings with undetermined signs. b) A diagram of a theta-curve of degree 4 in $\Omega_{3_1,4_1}^2$, constructed by choosing signs for the black crossings and gluing the parallel curves together in two nodes.\label{fig:double}}
\end{figure}
 
We call the process of choosing two points $n_{1}$, $n_{2}$ on a knot diagram and thereby dividing the knot into two arcs $\alpha_{1}$ and $\alpha_{2}$ a \textit{partition} of the knot diagram.

\begin{lemma}
\label{nonalt}
For all pairs of knots $K_{1}$, $K_{2}$, not both alternating, there is a partition $\alpha_{1}\cup\alpha_{2}=K_{1}\#K_{2}$ of any diagram of $K_{1}\#K_{2}$ such that for every $i\in\{1,2\}$ there is a crossing of $\alpha_{i}$ with itself.
\end{lemma}
\begin{proof}
Let $K_{1}$ and $K_{2}$ be knots not both alternating. Then $K_{1}\# K_{2}$ is not alternating. We pick a point $n_{1}$ on a diagram $D$ of $K_{1}\# K_{2}$ and consider the Gauss code starting at $n_{1}$ in an arbitrary direction. 

Let $n_{2}\neq n_{1}$ be a second point on the diagram and $\alpha_{1}$ the arc from $n_{1}$ to $n_{2}$ in the direction of the Gauss code.

Assume that $\alpha_{1}$ does not cross itself. This is equivalent to the position of $n_{2}$ on the knot diagram corresponding to a position in the Gauss code before an absolute value of a number appears for the second time in the Gauss code.

Similarly, $\alpha_{2}$, the other arc in the diagram, does not cross itself if and only if between the positions in Gauss code corresponding to $n_{2}$ and $n_{1}$ (in the direction of the Gauss code), no absolute value appears twice. 

Assume now that no matter where we place $n_{2}$ on the knot diagram, there is an $i=1,2$ such $\alpha_{i}$ does not cross itself. Then no matter where we split the Gauss code into two pieces, one piece will not contain any absolute value twice.

This means that every crossing must be visited once before the first instance of a crossing being visited for a second time, i.e. the first half of the Gauss code modulo signs reads $1, 2, \ldots, c(D)$. Now let $k\in\{1,2,\ldots,c(D)-1\}$ and assume that the crossing $k+1$ is visited the second time before $k$ is visited the second time. Then we could divide the Gauss code into two pieces, one of which contains both occurrences of the $k$ and $-k$ and the other both occurrences of $k+1$ and $-(k+1)$. Hence we found a partition where both $\alpha_{1}$ and $\alpha_{2}$ cross themselves.

If for every $k$ the crossing $k+1$ is visited the second time after crossing $k$ is visited the second time, then the sequence which is the absolute value of the Gauss code sequence is $1,2,\ldots,c(D),1,2,\ldots,c(D)$. It is easy to see that a knot that allows a diagram with such a Gauss code must be alternating, contradicting the assumption that $K_1$ and $K_2$ are not both alternating. 
\end{proof}

By Lemma \ref{nonalt} if $K_{1}$ and $K_{2}$ are not both alternating we can glue the link diagram $D_{n}$ of $n$ parallel copies of the diagram of $K_{1}\# K_{2}$ such that each of the edges of the resulting embedded graph crosses itself. Call the resulting diagram (with some undetermined crossing signs) $\tilde{D}$. We claim that now we can choose the signs of the crossings that are not determined yet in such a way that the resulting theta-curve of order $2n$ is in $\Omega_{K_{1},K_{2}}^n$, i.e. there are $n$ blue arcs $x_{i}$ and $n$ red arcs $z_{i}$ such that for all $i$ and $j$ the knot $x_{i}\cup z_{j}$ is $K_{1}\# K_{2}$ and none of $x_{i}\cup x_{i}$ and $z_{i}\cup z_{i}$ is the unknot or $K_1\# K_2\# K_1\# K_2$.

\begin{lemma}
We can choose the signs of the crossings of $\tilde{D}$ that are not determined yet in such a way that $\tilde{D}$ is a diagram of a theta-curve of degree $2n$ in $\Omega_{K_{1},K_{2}}^n$.
\end{lemma}

\begin{proof}
Note that by construction $x_{i}\cup z_{j}$, $i,j=1,2,\ldots,n$, is the original diagram $D$ of $K_{1}\# K_{2}$, where we deleted the information about the signs of the crossings. We can thus choose the signs of the crossings of $x_{i}$ with $z_{j}$ and the signs of crossings of $x_{i}$ and $z_{j}$ with themselves such that $x_{i}\cup z_{j}=K_{1}\# K_{2}$ for all $i$ and $j$.

We now need to determine the signs of the crossings of $x_{i}$ with $x_{j}$ and $z_{i}$ with $z_{j}$, $i\neq  j$. Note that $x_{i}$ and $x_{j}$ are two parallel arcs. So for each crossing between them, there is a cluster of four crossings, one of $x_{i}$ with itself, one of $x_{j}$ with itself (both of whose crossings have been already determined to carry identical signs) and two crossings of $x_{i}$ with $x_{j}$.

If for every such $4$-crossing we choose to give the crossings of $x_{i}$ and $x_{j}$ the same sign as the corresponding crossings of $x_{i}$ with itself and $x_{j}$ with itself, then $x_{i}$ and $x_{j}$ are two parallel curves glued together at their ends and hence $x_{i}\cup x_{j}$ is the unknot. We can move the ends, where $x_{i}$ and $x_{j}$ are glued together, through the knot to untie it. 

\begin{figure}[h!]\centering
\labellist
\pinlabel \textbf{a)} at 10 300
\pinlabel \textbf{b)} at 480 300
\pinlabel \textbf{c)} at 980 300
\endlabellist
\begin{subfigure}[t]{0.23\textwidth}
\includegraphics[width=\linewidth]{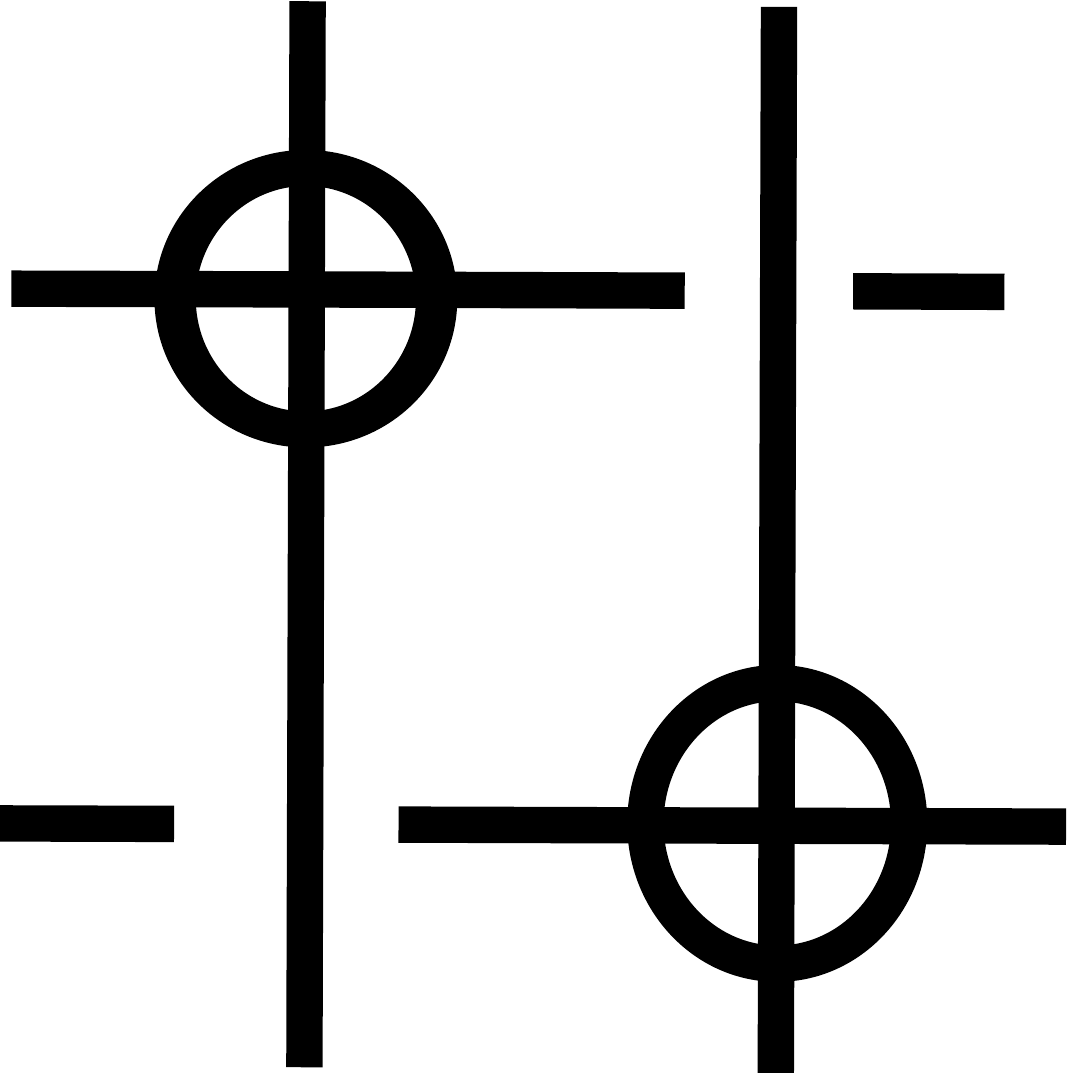}
\end{subfigure}\hfill
\begin{subfigure}[t]{0.28\textwidth}
\includegraphics[width=\linewidth, valign=b]{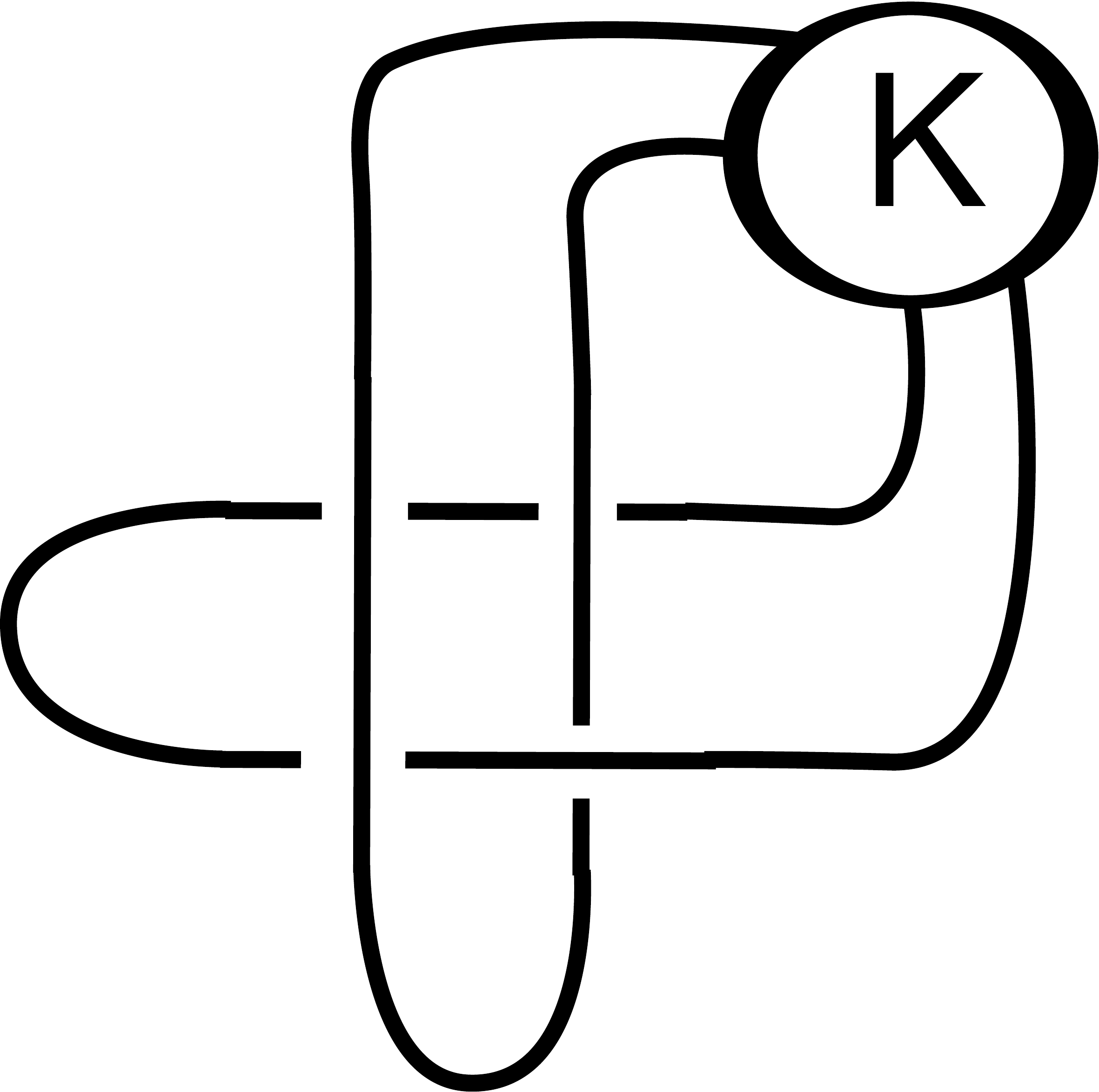}
\end{subfigure}\hfill
\begin{subfigure}[t]{0.28\textwidth}
\includegraphics[width=\linewidth, valign=b]{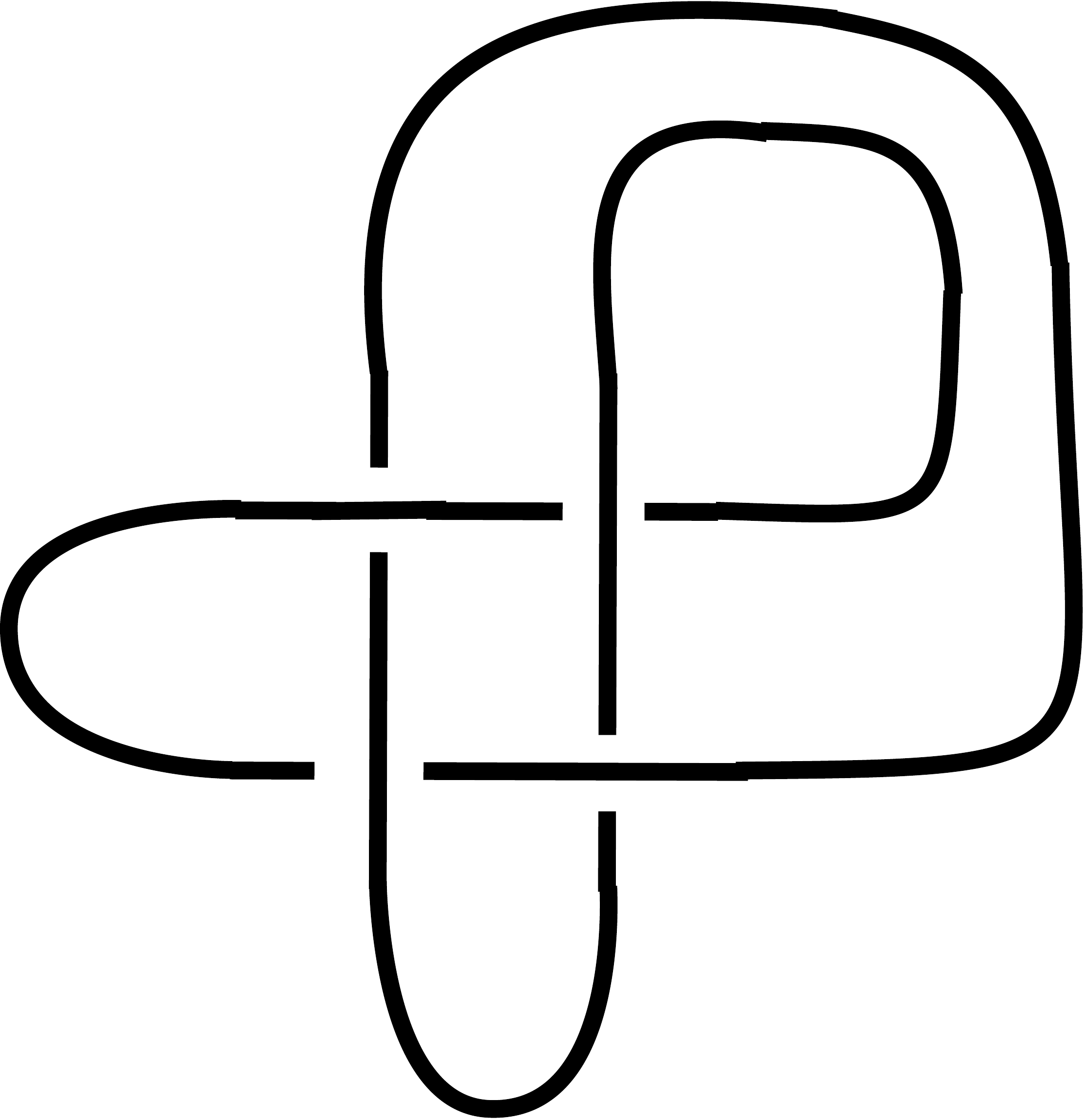}
\end{subfigure}
\caption{a) Doubling the strands turns every crossing into a 4-crossing, where two of the signs are given. Choosing the remaining signs results either in a diagram of a non-trivial Whitehead double of a knot $K$ (b) or in a trefoil (c).\label{fig:signs}}
\end{figure}

Instead we pick one such $4$-crossing, which exists by Lemma \ref{nonalt} for each pair $(x_i,x_j)$ and $(z_i,z_j)$. For all the others we distribute signs exactly as above, but for the one we picked we give the two crossings between $x_{i}$ and $x_{j}$ different signs. Then as we slide the ends of the curves through the knot as in the previous case, we obtain a diagram as in Figure \ref{fig:signs} b). It shows that the resulting knot is a Whitehead double of some knot $K$.

The only case where this Whitehead double is the unknot is if it is the untwisted Whitehead double of the unknot. In all other cases it is prime and therefore we have found a choice of signs for which $x_{i}\cup x_{j}$ is neither the unknot nor $K_1\# K_2\# K_1\# K_2$.

If $K$ is the unknot and the Whitehead double is untwisted, we can change one of the crossings in the $4$-crossing that we picked, so that now the two crossings between $x_{i}$ and $x_{j}$ both have different signs from the crossings of $x_{i}$ and $x_{j}$ with themselves. In this case the diagram that we obtain is the trefoil (Figure \ref{fig:signs} c)). 

Therefore, we can always choose the signs of the crossings in such a way that $x_{i}\cup z_{j}=K_{1}\# K_{2}$ and $x_{i}\cup x_{j}$ and $z_{i}\cup z_{j}$ are neither the unknot nor $K_1\# K_2\# K_1\# K_2$ for all $i,j=1,2,\ldots,n$.    
\end{proof}

Note that for alternating knots the additivity of the crossing number is known, so the next proposition follows from the previous lemmas and the opening remarks to this section.
\begin{proposition}
\label{square}
For every $n\in\mathbb{Z}_{>0}$ we have $c(K_{1}\# K_{2})\geq\frac{1}{n^2} c(\Omega_{K_{1},K_{2}}^{n})$.
\end{proposition}

As mentioned before, Proposition \ref{square} opens up the possibility of finding lower bounds for $c(K_{1}\# K_{2})$ by finding low $n$ such that $c(\Omega_{K_{1},K_{2}}^n)=n(c(K_{1})+c(K_{2}))$, since then $c(K_{1}\# K_{2})\geq \tfrac{1}{n} (c(K_{1})+c(K_{2}))$. 

Note that the $\Gamma$-graph associated to the constructed diagram $\tilde{D}$ (after the signs have been assigned) does not contain a bicoloured triangle. The next corollary follows directly.

\begin{corollary}
Let $n\in\mathbb{Z}_{\geq 2}$ such that every diagram $D$ of any theta-curve $\theta\in\Omega_{K_1,K_2}^n$ such that $\Gamma(D)$ does not have any bicoloured triangles satisfies $c(D)\geq n(c(K_1)+c(K_2))$. Then $c(K_1\# K_2)\geq \frac{1}{n}(c(K_1)+c(K_2))$.
\end{corollary}

Lemma \ref{largek} shows that such values for $n$ exist.
For example, the value of $n=2(c(K_1)+c(K_2)-c(K_1\# K_2))+1$ found in Lemma \ref{largek} gives $c(K_1\# K_2)\geq \frac{c(K_1)+c(K_2)}{2(c(K_1)+c(K_2)-c(K_1\# K_2))+1}$, which is trivial. However, if we could improve on the value of $n$, then we would obtain a new lower bound for $c(K_1\# K_2)$.

\section{Other graphs}
\label{sec:outlook}

In this section we consider graphs with more than two nodes starting with the example graph $\oplus$ with four 3-valent vertices connected by edges in a circle and one 4-valent vertex that is connected to every other vertex by an edge. We want to think of this graph as two theta-graphs glued together in a neighbourhood of one of their vertices.

The set of theta-curves also comes with a notion of connected sum. We can orient the edges of a theta-curve such that one of its vertices is a source $n_{1}$ and the other is a sink $n_{2}$. Then the connected sum of two theta-curves, $\theta_1$ and $\theta_2$, is formed by deleting a neighbourhood of $n_{2}$ of $\theta_{1}$ and a neighbourhood of $n_{1}$ of $\theta_{2}$ and gluing the theta-curves together on the open ends of their arcs, joining arcs with the same labels $x$, $y$ and $z$ respectively. 
In order to make this a natural operation we should consider two embedded graphs to be equivalent iff they are related by an ambient isotopy that does not change the clockwise order in which the arcs meet the node.

Note that the connected sum commutes with tying knots into one of the arcs, in particular $\theta_{K_1,K_2}\# \theta_{K_3,K_4}=\theta_{K_1\#K_3,K_2\# K_4}$. This means that if the crossing number of theta-curves is additive under connected sum, then the crossing number of knots is also additive (simply take $K_{2}$ and $K_4$ to be the unknot).

A fundamental concept of Section \ref{sec:proof} can now easily be generalised to $\oplus$ (and in fact beyond). The step from knots to theta-curves in Section \ref{sec:proof} is adding an extra arc, which we will think  of as adding the part of the knot (or in this case the theta-curve) that was deleted in the process of the connected sum. In the case of the connected sum of two theta-curves adding the deleted part back in results in $\oplus$.

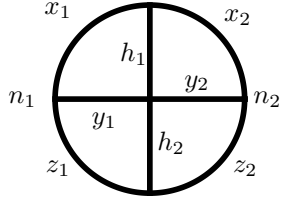
\begin{figure}\centering
\labellist
\pinlabel $x_1$ at 68 299
\pinlabel $x_2$ at 310 289
\pinlabel $y_1$ at 130 149
\pinlabel $y_2$ at 255 200
\pinlabel $z_1$ at 68 85
\pinlabel $z_2$ at 320 85
\pinlabel $n_1$ at 20 179
\pinlabel $n_2$ at 350 179
\pinlabel $h_1$ at 170 240
\pinlabel $h_2$ at 220 120
\endlabellist
\includegraphics[height=3.5cm]{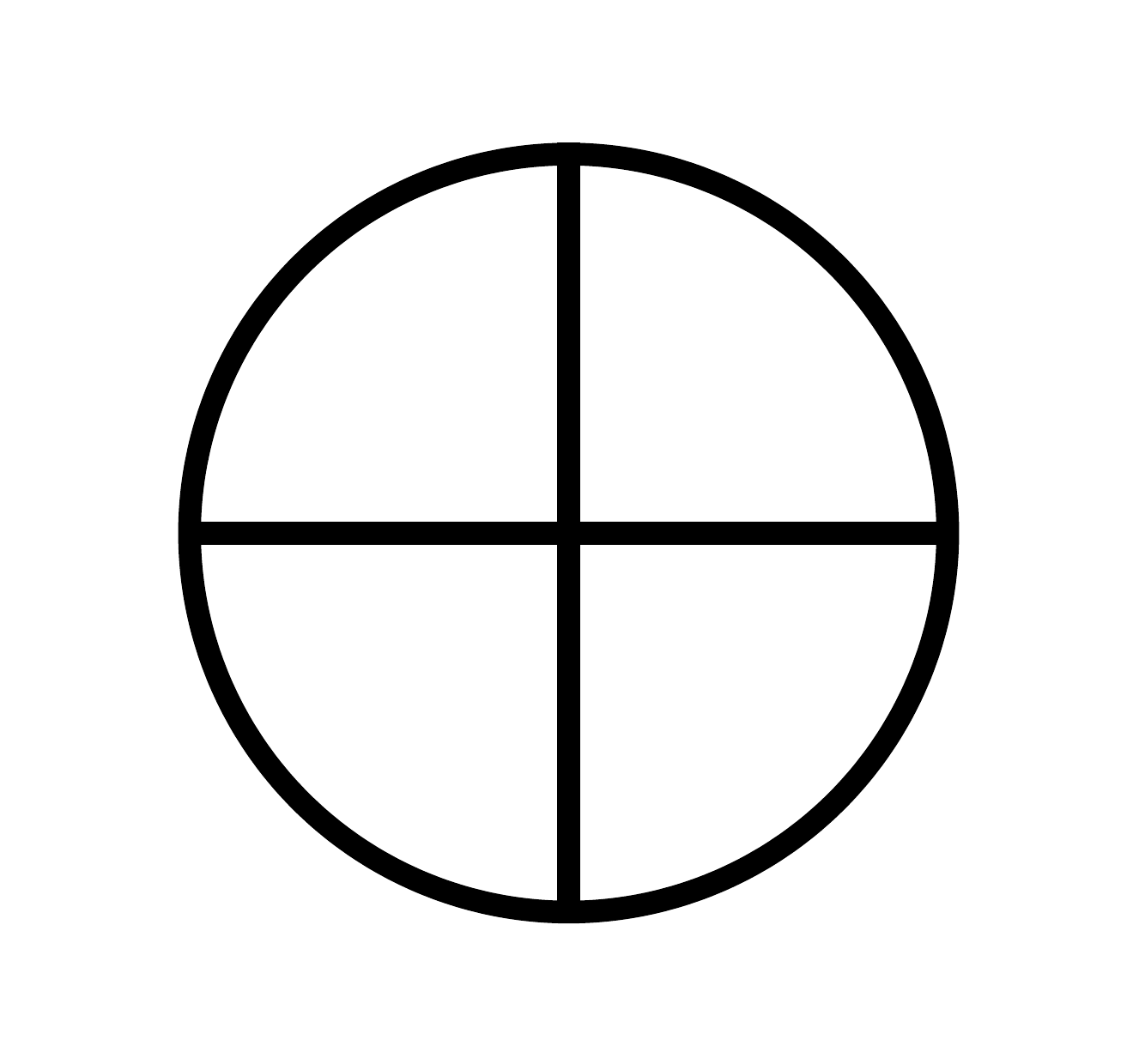}
\caption{The planar embedding of the $\oplus$-graph with labelled edges.\label{fig:def_oplus}}
\end{figure}

We label the edges of this graph as follows: We fix one of the 3-valent vertices $n_{1}$ and denote the edges connected to $n_1$ by $x_{1}$, $y_{1}$ and $z_{1}$. The only 3-valent vertex that is not connected to $n_{1}$ is called $n_{2}$ and edges connecting to $n_{2}$ have labels $x_{2}$, $y_{2}$ and $z_{2}$ such that $x_{1}$ and $x_{2}$ (and similarly $y_{1}$ and $y_{2}$ as well as $z_{1}$ and $z_{2}$) meet at a vertex. The two edges that are left are called $h_{1}$ and $h_{2}$.

Consider now an embedding of $\oplus$ where a copy of $K_{1}$ is tied into $x_{1}$ and $z_{2}$ of the planar $\oplus$ and a copy of $K_{2}$ is tied into each of the edges $z_{1}$ and $x_{2}$, which we denote by $\oplus_{K_1,K_2}$.
Then for each $i\in\{1,2\}$ deleting $x_{i}$, $y_{i}$ and $z_{i}$ results in a diagram of a theta-curve $\theta_{K_1,K_2}$. In other words
\begin{equation}
c(\oplus_{K_1,K_2})+h_1h_1+h_2h_2+h_1h_2\geq 2c(\theta_{K_1,K_2})+\sum_{(k,l)\in\{x,y,z\}^2}k_1l_2.
\end{equation}

On the other hand, deleting the edges $h_1$ and $h_2$ results in the theta-curve $\theta_{K_1\# K_2,K_1\# K_2}$. We thus have a situation that is similar to that of Section 2, where
\begin{equation}
c(\oplus_{K_1,K_2})\geq c\left(\theta_{K_1\# K_2,K_1\# K_2}\right),
\end{equation}
and equality is equivalent to $c(\theta_{K_1\# K_2,K_1\# K_2})=2c(\theta_{K_1,K_2})$. Since 
\begin{equation}
c(\theta_{K_1\# K_2,K_1\# K_2})\leq 2c(K_{1}\# K_2),
\end{equation} 
this then is equivalent to $c(K_1\# K_2)=c(\theta_{K_1,K_2})$ and by Proposition \ref{prop1} to the additivity of the crossing number.

Analogously, we can define the connected sum of two theta-curves of any degree (cf. Figure \ref{fig:oplus_sum}). 

\begin{figure}[tb]\centering
\labellist
\pinlabel \textbf{a)} at 20 850
\pinlabel \textbf{b)} at 1050 850
\endlabellist
\begin{subfigure}[t]{0.45\textwidth}
\includegraphics[width=\linewidth]{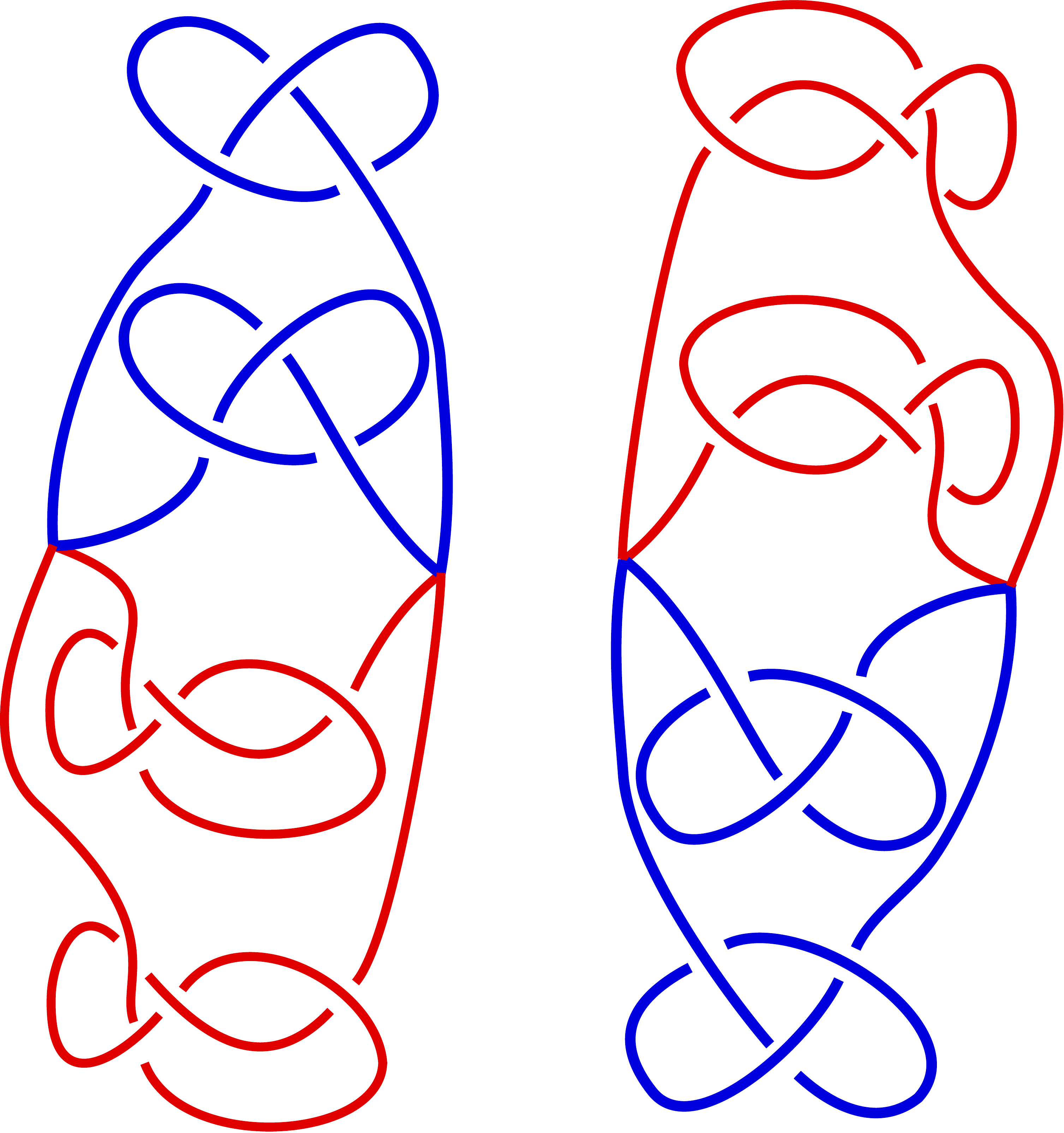}
\end{subfigure}\hfill
\begin{subfigure}[t]{0.45\textwidth}
\includegraphics[width=\linewidth, valign=b]{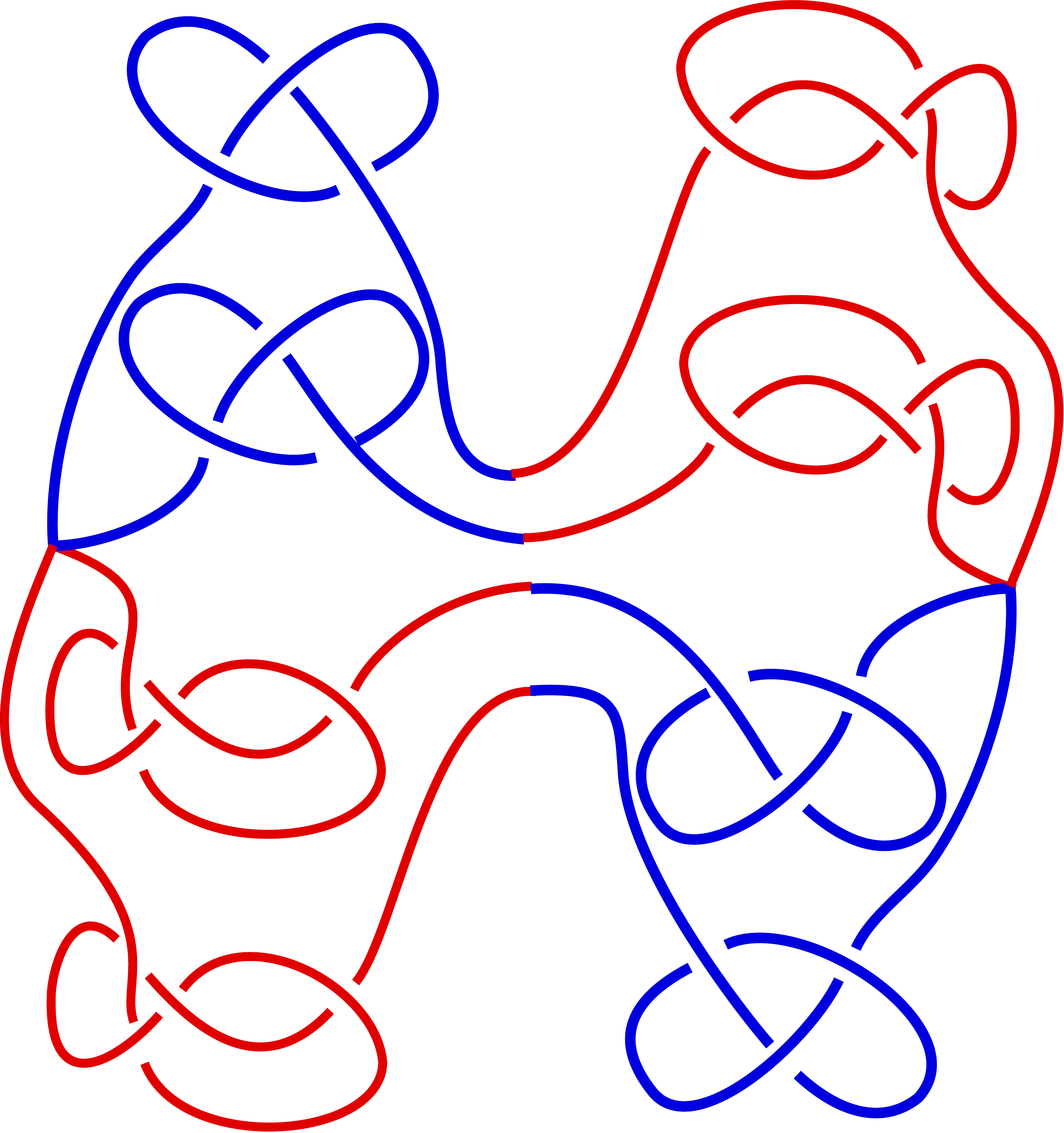}
\end{subfigure}
\caption{Two $4$-theta-curves (a) are added with the connected sum operation (b).\label{fig:oplus_sum}}
\end{figure}

\begin{figure}[h!]\centering
\labellist
\pinlabel \textbf{a)} at 20 900
\pinlabel \textbf{b)} at 1750 900
\pinlabel \textbf{c)} at 20 -430
\pinlabel \textbf{d)} at 1750 -430
\pinlabel \textbf{e)} at 20 -1830
\pinlabel \textbf{f)} at 1550 -1830
\endlabellist
\begin{subfigure}[t]{0.45\textwidth}
\vspace{-3cm}
\includegraphics[width=\linewidth]{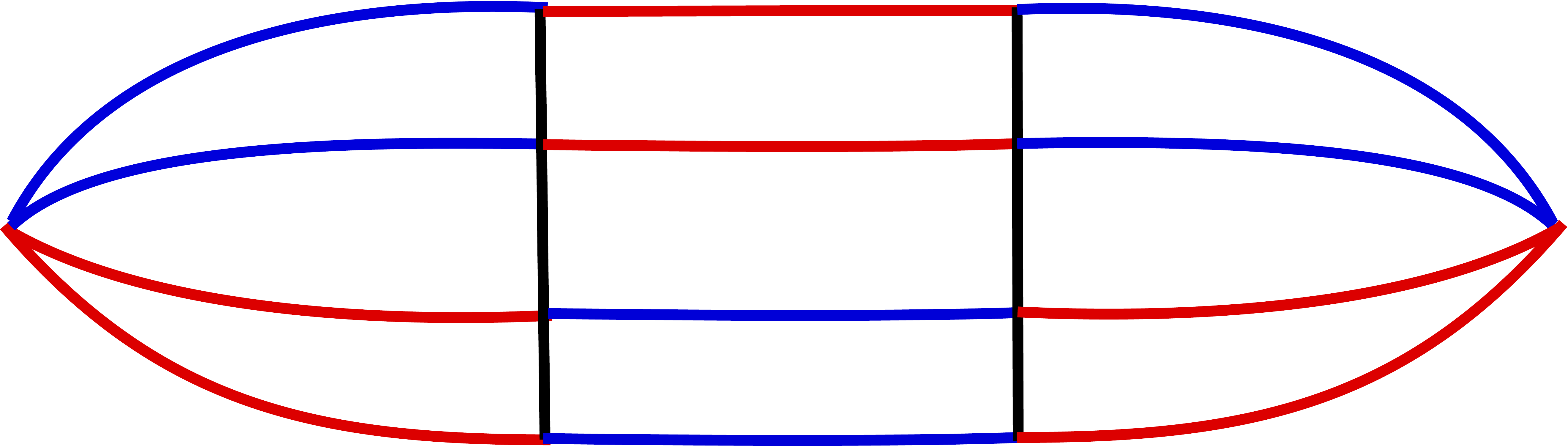}
\end{subfigure}\hfill
\begin{subfigure}[t]{0.45\textwidth}
\includegraphics[width=\linewidth]{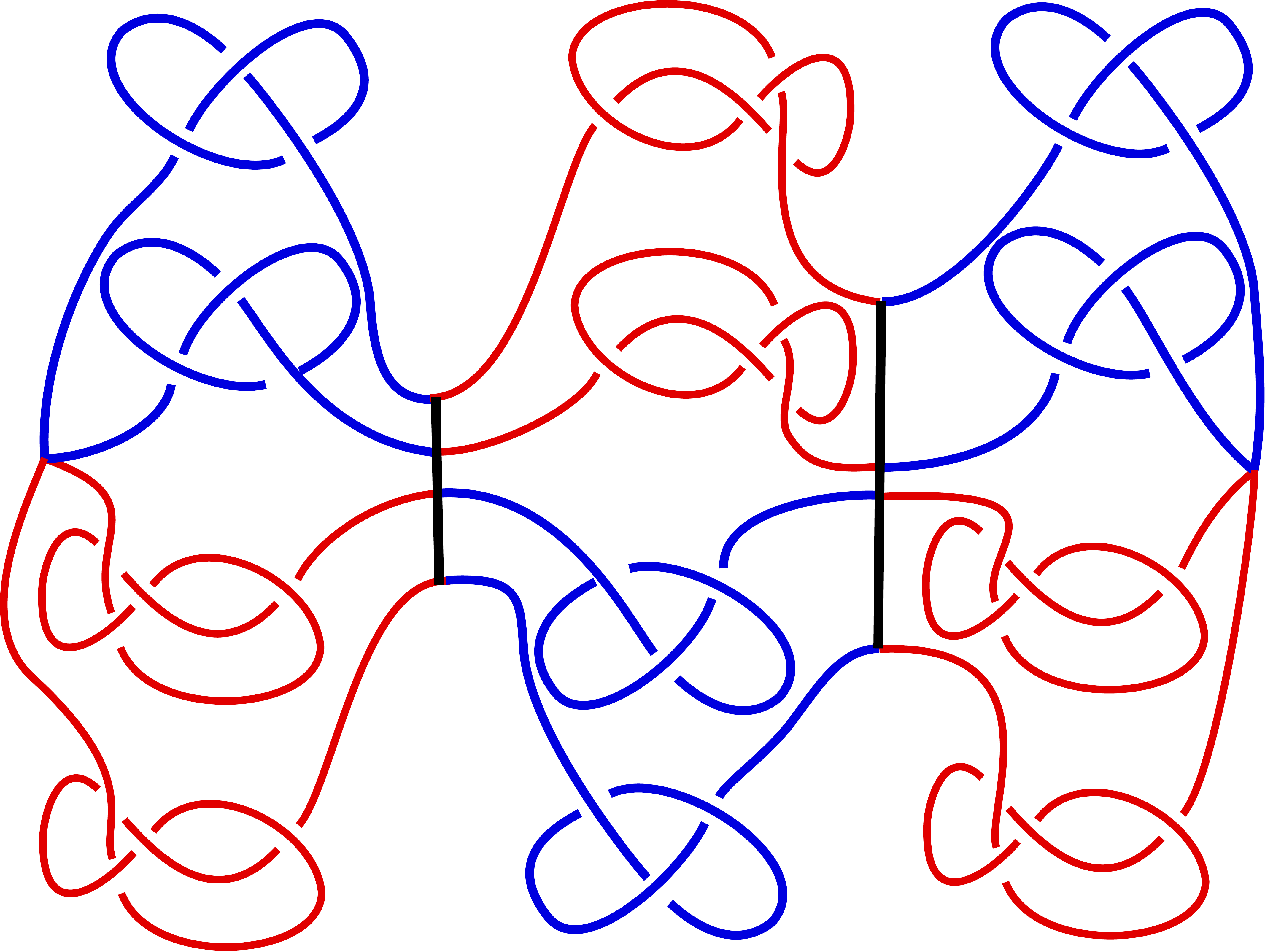}
\end{subfigure}
\begin{subfigure}[t]{0.45\textwidth}
\vspace{0.5cm}
\hspace{-0.5cm}
\includegraphics[width=\linewidth]{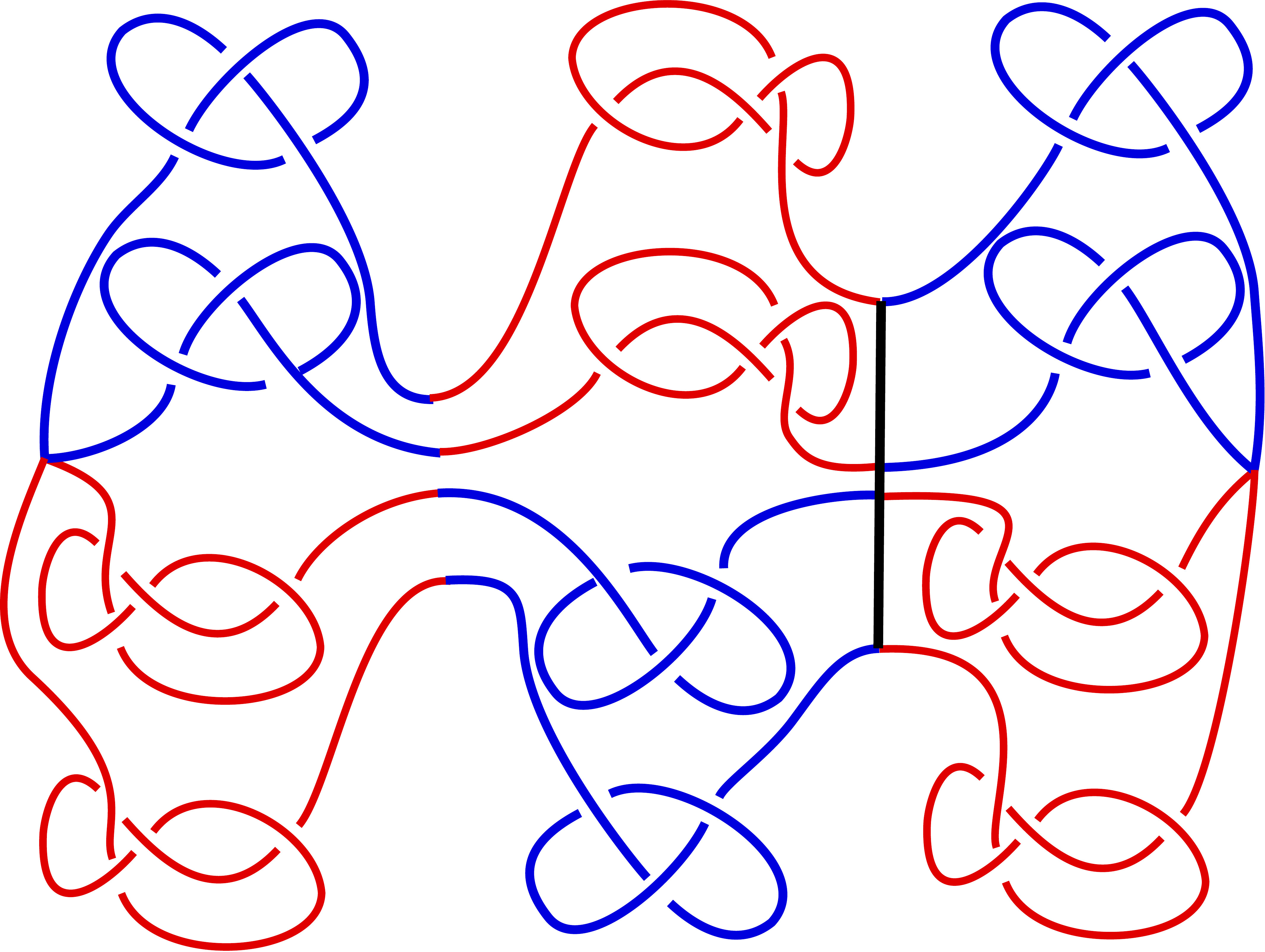}
\end{subfigure}
\begin{subfigure}[t]{0.45\textwidth}
\vspace{0.5cm}
\hspace{0.5cm}
\includegraphics[width=\linewidth]{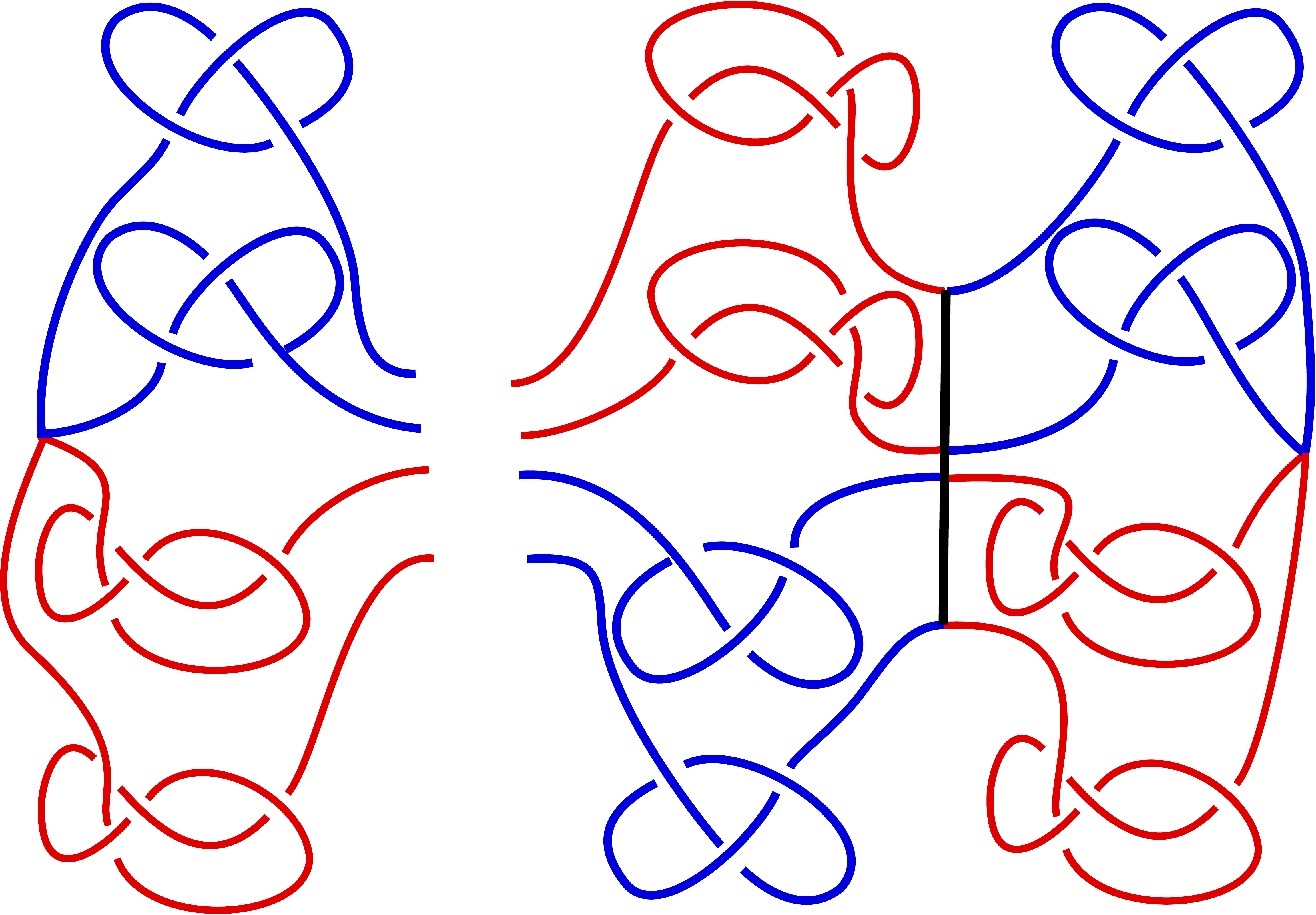}
\end{subfigure}
\begin{subfigure}[t]{0.4\textwidth}
\vspace{1.5cm}
\hspace{-0.5cm}
\includegraphics[width=\linewidth]{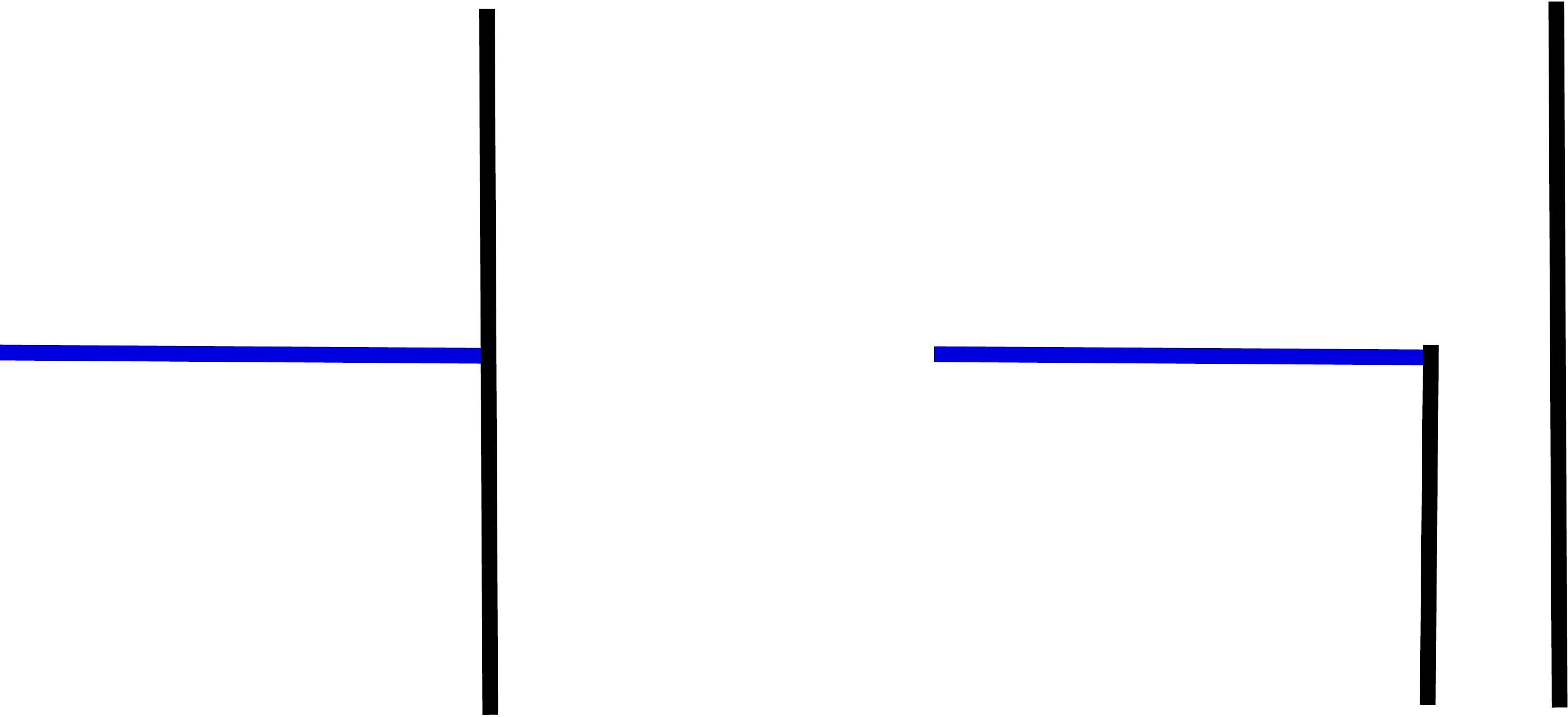}
\end{subfigure}
\begin{subfigure}[t]{0.55\textwidth}
\vspace{0.5cm}
\hspace{0.5cm}
\includegraphics[width=\linewidth]{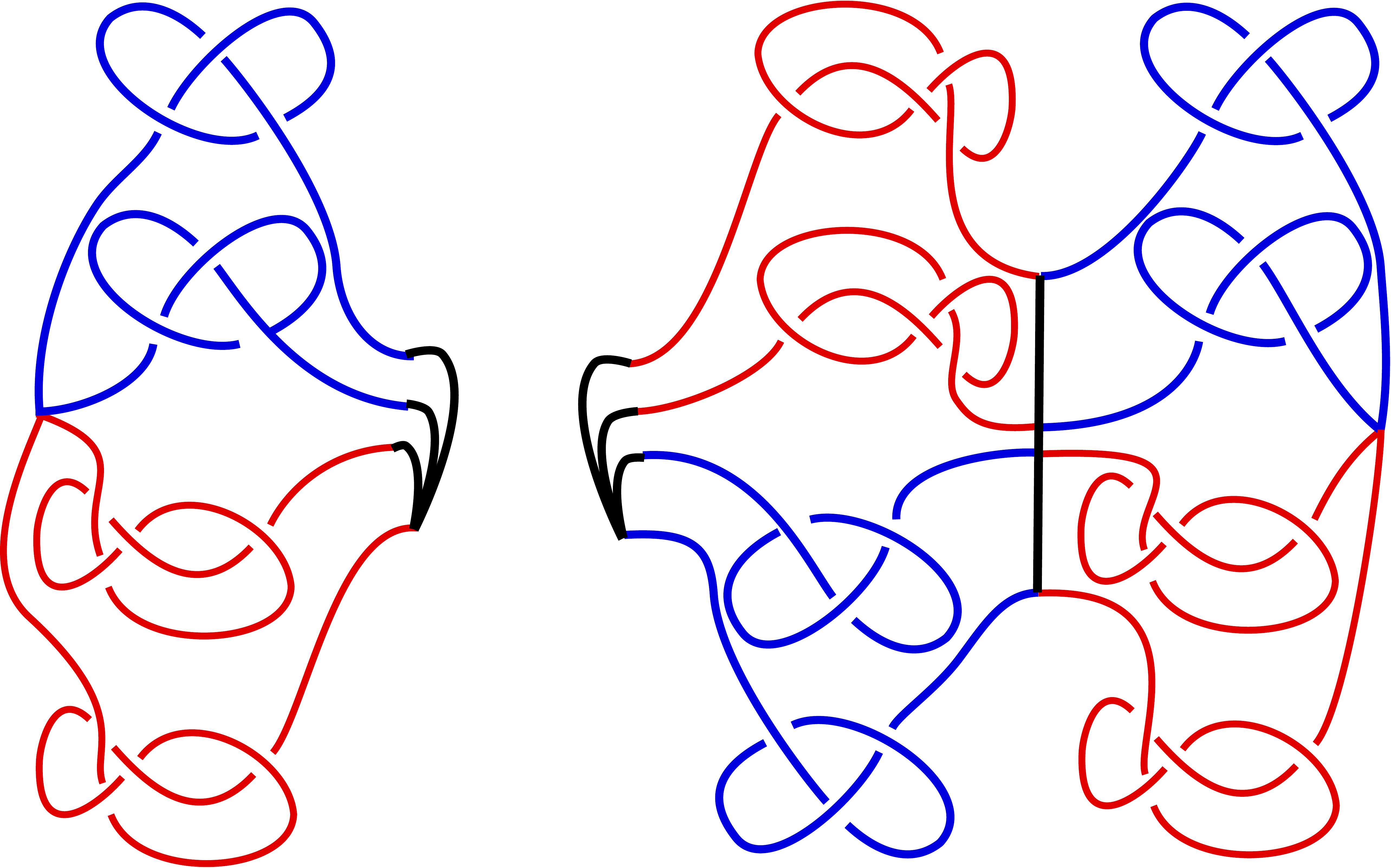}
\end{subfigure}
\caption{a) $\oplus^{2,2}$. b) $\oplus_{K_1=3_1,K_2=4_1}^{2,2}$ is an embedding of the graph $\oplus^{2,2}$. c) A diagram of $\mathrm{G}_{3_1,4_1}^{2,2,1}$. d) Cutting a diagram of $\oplus_{3_1,4_1}^{2,2}$ along the first vertical edge. e) Resolution of the nodes in c) to close d) to a diagram of $\oplus_{3_1,4_1}^{2,0}$ and $\oplus_{3_1,4_1}^{2,1}$. f) The resulting diagram of $\oplus_{3_1,4_1}^{2,0}$ and $\oplus_{3_1,4_1}^{2,1}$.\label{fig:oplus}}
\end{figure}

Let $\oplus^{n,k}$ denote the graph (as in Figure \ref{fig:oplus}a)) with $k$ vertical edges and $2n$ rows of horizontal edges.
Let $\oplus_{K_{1},K_2}^{n,k}$ denote the spatial graph that is obtained from the planar embedding of $\oplus^{n,k}$ by tying in each column $n$ of the horizontal edges into $K_1$ and the remaining $n$ horizontal edges into $K_2$, such that at every node an arc with a $K_1$ meets an arc with a $K_2$ (cf. Figure \ref{fig:oplus}b)).
We denote by $\mathrm{G}_{K_1,K_2}^{n,k,i}$ the graph (cf. Figure \ref{fig:oplus}c)) that results from $\oplus_{K_1,K_2}^{n,k}$ by deleting the $i$th vertical edge.
Note that $\oplus_{K_1\# K_2,K_1\# K_2}^{n,0}=\mathrm{G}_{K_1,K_2}^{n,1,1}$.

%Again, Lemma \ref{jordan}, Proposition \ref{prop1} and the arguments for $\oplus_{K_1,K_2}$ generalize to the following lemma.
\begin{lemma}
\label{general}
%For all positive integers $n$, $k$ and $i$ we have
%\begin{equation}
%\label{eq:graph}
%c\left(\oplus_{K_1,K_2}^{n,k}\right)\geq \frac{1}{2} \left(c\left(\oplus_{K_1,K_2}^{n,i-1}\right)+c\left(\oplus_{K_1,K_2}^{n,k-i}\right)+c\left(\mathrm{G}_{K_1,K_2}^{n,k,i}\right)\right).
%\end{equation}
For all positive integers $n$, $k$ and $i$ we have
\begin{equation}
\label{eq:graph}
c\left(\oplus_{K_1,K_2}^{n-i-1}\right)+c\left(\oplus_{K_1,K_2}^{n,k-i}\right)\geq
c\left(\oplus_{K_1,K_2}^{n,k}\right)\geq c\left(\mathrm{G}_{K_1,K_2}^{n,k,i}\right).
\end{equation}
Furthermore, if $c\left(\oplus_{K_1,K_2}^{n,k}\right)= c\left(\mathrm{G}_{K_1,K_2}^{n,k,i}\right)$, then 
\begin{equation}
\label{eq:cond}
c\left(\oplus_{K_1,K_2}^{n,k}\right)= c\left(\oplus_{K_1,K_2}^{n,i-1}\right)+c\left(\oplus_{K_1,K_2}^{n,k-i}\right).
\end{equation}

%If Equation (\ref{eq:graph}) is an equality, then 
%\begin{equation}
%c\left(\oplus_{K_1,K_2}^{n,k}\right)= c\left(\oplus_{K_1,K_2}^{n,i-1}\right)+c\left(\oplus_{K_1,K_2}^{n,k-i}\right)=c\left(\mathrm{G}_{K_1,K_2}^{n,k,i}\right)
%\end{equation}
\end{lemma}
%The proof is analogous to the one of Lemma \ref{jordan} and Proposition \ref{prop1}.
\begin{proof}
Equation (\ref{eq:graph}) is almost immediate. We can form the connected sum of $\oplus_{K_1,K_2}^{n-i-1}$ and $\oplus_{K_1,K_2}^{n,k-i}$ using their minimal diagrams. Since this process involves deleting a small neighbourhood of two vertices, we can add an unknotted arc to form a diagram of $\oplus_{K_1,K_2}^{n,k}$ with $c\left(\oplus_{K_1,K_2}^{n-i-1}\right)+c\left(\oplus_{K_1,K_2}^{n,k-i}\right)$ many crossings. Deleting the $i$th vertical edge in the minimal diagram of $\oplus_{K_1,K_2}^{n,k}$ results in a diagram of $\mathrm{G}_{K_1,K_2}^{n,k,i}$, which proves the inequality on the right hand side of Equation (\ref{eq:graph}).

If $c\left(\oplus_{K_1,K_2}^{n,k}\right)= c\left(\mathrm{G}_{K_1,K_2}^{n,k,i}\right)$, then the $i$th vertical edge in the minimal diagram of $\oplus_{K_1,K_2}^{n,k}$ is not involved in any crossings, neither with itself nor with any other edge of the spatial graph. Otherwise deleting the $i$th vertical edge in the minimal diagram of $\oplus_{K_1,K_2}^{n,k}$ would result in a diagram of $\mathrm{G}_{K_1,K_2}^{n,k,i}$ with strictly less than $c\left(\oplus_{K_1,K_2}^{n,k}\right)= c\left(\mathrm{G}_{K_1,K_2}^{n,k,i}\right)$ crossings. We can therefore cut $\oplus_{K_1,K_2}^{n,k}$ along the $i$th arc to obtain two spatial graphs (as in Figure \ref{fig:oplus}d)), whose open ends can be joined in one vertex without introducing any crossings.

%%%%%%%%%%%%%
This can be seen as follows. The $i$th vertical edge in the minimal diagram of $\oplus_{K_1,K_2}^{n,k}$ has $2n$ vertices on it, 2 of valency 3 and $2(n-1)$ of valency 4. We cut the diagram along the $i$th vertical edge and now want to connect the open ends of the remaining diagram without introducing extra crossings.
We start with one of the endpoints of the deleted edge, i.e.~one of the nodes that had valency 3 in $\oplus_{K_1,K_2}^{n,k}$. We follow the deleted $i$th vertical edge until we encounter the next node. We resolve this node as in Figure 11e) in a similar fashion to the proof of Proposition \ref{resolve}. Now we have two parallel curves that follow the deleted $i$th vertical edge until the next vertex, that also gets resolved accordingly. This process continues until all $2n-1$ parallel arcs are glued to the last remaining open end.
It is clear that this results in a diagram of $\oplus_{K_1,K_2}^{n,i-1}$ and of $\oplus_{K_1,K_2}^{n,k-i}$ as in Figure \ref{fig:oplus}f). Furthermore, this closing procedure does not lead to any new crossings, since all added arcs are parallel to the deleted $i$th vertical edge, which was not involved in any crossings.
%%%%%%%%%%%%%%%%%

This results in a diagram of $\oplus_{K_1,K_2}^{n,i-1}$ and of $\oplus_{K_1,K_2}^{n,k-i}$, which shows that $c\left(\oplus_{K_1,K_2}^{n,k}\right)\geq c\left(\oplus_{K_1,K_2}^{n,i-1}\right)+c\left(\oplus_{K_1,K_2}^{n,k-i}\right)$. Equation (\ref{eq:cond}) then follows from Equation (\ref{eq:graph}).  
\end{proof}

Similar arguments apply to the spatial graph $\mathrm{G}_{K_1,K_2}^{n,k,i}$ as well.
 
\begin{lemma}
For all positive integers $n$, $k$ and $i\neq (k+1)/2$ we have
\begin{equation}
\label{eq:genki}
c\left(\mathrm{G}_{K_1,K_2}^{n,k,i}\right)\leq c\left(\oplus_{K_1,K_2}^{n,\min\{i-1,k-i\}}\right)+ c\left(\mathrm{G}_{K_1,K_2}^{n,k-1-\min\{i-1,k-i\},s}\right),
\end{equation}
where
\begin{equation}
\label{eq:s}
s=\begin{cases}i &\text{ if }i-1<k-i,\\
i-(k-i)-1 &\text{ if }i-1>k-i
\end{cases}.
\end{equation}
\end{lemma}
\begin{proof}
First note that the case of $i-1=k-i$ cannot occur, since then $i=(k+1)/2$. Hence $s$ is well-defined.

We can form the connected sum of $\oplus_{K_1,K_2}^{n,\min\{i-1,k-i\}}$ and $\mathrm{G}_{K_1,K_2}^{n,k-1-\min\{i-1,k-i\},s}$ using their minimal diagrams. Since the connected sum involves deleting neighbourhoods of two nodes, we can add an extra arc to obtain a diagram of $\mathrm{G}_{K_1,K_2}^{n,k,i}$ without adding any extra crossings. Therefore the minimal crossing number of $\mathrm{G}_{K_1,K_2}^{n,k,i}$ is at most
\begin{equation}
c\left(\oplus_{K_1,K_2}^{n,\min\{i-1,k-i\}}\right)+ c\left(\mathrm{G}_{K_1,K_2}^{n,k-1-\min\{i-1,k-i\},s}\right).
\end{equation}
\end{proof}
 
Furthermore, Proposition \ref{prop1} generalizes to the following statement.
\begin{proposition}
\label{rec}
If there exist positive integers $n$, $k$ and $m$ such that $c\left(\oplus_{K_1,K_2}^{n,k}\right)=c\left(\mathrm{G}_{K_1,K_2}^{n,k,i}\right)$ holds for $i=k/m$ or $i=\tfrac{m-1}{m}k+1$, then $c(K_1\# K_2)=c(K_1)+c(K_2)$.
\end{proposition}

\begin{proof}
%We consider three different cases, firstly the case of $i=1$ or $i=n$, then $i=(k+1)/2$ and lastly $i\notin\{1,k,(k+1)/2\}$.
We start with $m=1$, so $i=1$ or $i=k$. We assume that $i=1$. The case of $i=k$ can be proven analogously. By Lemma \ref{general} $c\left(\oplus_{K_1,K_2}^{n,k}\right)=c\left(\mathrm{G}_{K_1,K_2}^{n,k,1}\right)$ implies that \begin{equation}
c(\oplus_{K_1,K_2}^{n,k})=c(\oplus_{K_1,K_2}^{n,0})+c(\oplus_{K_1,K_2}^{n,k-1})=c(\mathrm{G}_{K_1,K_2}^{n,k,1}).
\end{equation}
Using Equation (\ref{eq:genki}) with $i=1$,
\begin{equation}
c\left(\mathrm{G}_{K_1,K_2}^{n,k,1}\right)\leq c\left(\oplus_{K_1,K_2}^{n,0}\right)+c\left(\mathrm{G}_{K_1,K_2}^{n,k-1,1}\right), 
\end{equation}
we get 
\begin{equation}
c\left(\oplus_{K_1,K_2}^{n,k-1}\right)\leq c\left(\mathrm{G}_{K_1,K_2}^{n,k-1,1}\right),
\end{equation}
which by Lemma \ref{general} implies
\begin{equation}
c\left(\oplus_{K_1,K_2}^{n,k-1}\right)=c\left(\mathrm{G}_{K_1,K_2}^{n,k-1,1}\right).
\end{equation}
We have just shown that if $c\left(\oplus_{K_1,K_2}^{n,k}\right)=c\left(\mathrm{G}_{K_1,K_2}^{n,k,1}\right)$, then the same equality holds for $k-1$.
Iterating this process shows that
\begin{equation}
\label{eq:11}
c\left(\oplus_{K_1,K_2}^{n,1}\right)=c\left(\mathrm{G}_{K_1,K_2}^{n,1,1}\right)=2c\left(\oplus_{K_1 ,K_2}^{n,0}\right). 
\end{equation}
Note that $\mathrm{G}_{K_1,K_2}^{n,1,1}=\theta_{K_1\# K_2,K_1\# K_2}^n$, so in particular 
\begin{equation}
c\left(\mathrm{G}_{K_1,K_2}^{n,1,1}\right)\leq 2nc(K_1\# K_2).
\end{equation}
Using Equation (\ref{eq:11}) we obtain 
\begin{equation}
c\left(\oplus_{K_1,K_2}^{n,0}\right)\leq nc(K_1\# K_2).
\end{equation}
Note that $\oplus_{K_1,K_2}^{n,0}=\theta_{K_1,K_2}^n$ and $c(\theta_{K_1,K_2}^n)\geq nc(K_1\# K_2)$ (by Corollary \ref{mixed}) and thus we have $c(\theta_{K_1,K_2}^n)= nc(K_1\# K_2)$, which by Corollary \ref{mixed} implies that $c(K_1\# K_2)=c(K_1)+c(K_2)$.

%Next we assume that $i=(k+1)/2$. Lemma \ref{general} says that Equation (\ref{eq:cond}) implies
%\begin{equation}
%c\left(\oplus_{K_1,K_2}^{n,k}\right)=2c\left(\oplus_{K_1,K_2}^{n,(k-1)/2}\right)=c\left(\mathrm{G}_{K_1,K_2}^{n,k,(k+1)/2}\right).
%\end{equation} 
%Since
%\begin{equation}
%\label{eq:n}
%2c\left(\oplus_{K_1,K_2}^{n,(k-1)/2}\right)\geq c\left(\oplus_{K_1,K_2}^{n,0}\right)+c\left(\oplus_{K_1,K_2}^{n,(k-1)/2-1}\right)+c\left(\mathrm{G}_{K_1,K_2}^{n,(k-1)/2,(k-1)/2}\right)
%\end{equation}
%and 
%\begin{equation}
%c\left(\mathrm{G}_{K_1,K_2}^{n,k,(k+1)/2}\right)\leq c\left(\mathrm{G}_{K_1,K_2}^{n,(k+1)/2,(k+1)/2}\right)+c\left(\oplus_{K_1,K_2}^{n,(k-1)/2}\right),
%\end{equation}
%we obtain
%\begin{equation}
%c\left(\mathrm{G}_{K_1,K_2}^{n,(k-1)/2,(k-1)/2}\right)+c\left(\oplus_{K_1,K_2}^{n,0}\right)\leq c\left(\mathrm{G}_{K_1,K_2}^{n,(k+1)/2,(k+1)/2}\right),
%\end{equation}
%which implies 
%\begin{equation}
%c\left(\mathrm{G}_{K_1,K_2}^{n,(k-1)/2,(k-1)/2}\right)+c\left(\oplus_{K_1,K_2}^{n,0}\right)=c\left(\mathrm{G}_{K_1,K_2}^{n,(k+1)/2,(k+1)/2}\right).
%\end{equation}

%This however means that the inequality in Equation (\ref{eq:n}) is an equality. Note that this equality satisfies the conditions of the first case and thus implies that $c(K_1\# K_2)=c(K_1)+c(K_2)$.

Now we assume that we have $c\left(\oplus_{K_1,K_2}^{n,k}\right)=c\left(\mathrm{G}_{K_1,K_2}^{n,k,i}\right)$ with $i=k/m$ or $i=\tfrac{m-1}{m}k+1$ for some $m>1$. In particular, $i\neq \tfrac{k+1}{2}$.

It follows again from Lemma \ref{general} that
\begin{equation}\label{eq:bcd}
c\left(\oplus_{K_1,K_2}^{n,k}\right)=c\left(\oplus_{K_1,K_2}^{n,i-1}\right)+c\left(\oplus_{K_1,K_2}^{n,k-i}\right)=c\left(\mathrm{G}_{K_1,K_2}^{n,k,i}\right).
\end{equation}

Combining Equation (\ref{eq:bcd}) and Equation (\ref{eq:genki}) gives
\begin{align}
c\left(\oplus_{K_1,K_2}^{n,\min\{i-1,k-i\}}\right)&+c\left(\oplus_{K_1,K_2}^{n,\max\{i-1,k-i\}}\right)=c\left(\oplus_{K_1,K_2}^{n,i-1}\right)+c\left(\oplus_{K_1,K_2}^{n,k-i}\right)\nonumber\\
&\leq c\left(\oplus_{K_1,K_2}^{n,\min\{i-1,k-i\}}\right)+c\left(\mathrm{G}_{K_1,K_2}^{n,k-1-\min\{i-1,k-i\},s}\right),
\end{align}
with $s$ as in Equation (\ref{eq:s}). 

Canceling $c\left(\oplus_{K_1,K_2}^{n,\min\{i-1,k-i\}}\right)$ leaves us with
\begin{equation}
c\left(\oplus_{K_1,K_2}^{n,\max\{i-1,k-i\}}\right)\leq c\left(\mathrm{G}_{K_1,K_2}^{n,k-1-\min\{i-1,k-i\},s}\right),
\end{equation}
which implies
\begin{equation}
c\left(\oplus_{K_1,K_2}^{n,\max\{i-1,k-i\}}\right)= c\left(\oplus_{K_1,K_2}^{n,k-1-\min\{i-1,k-i\}}\right)=c\left(\mathrm{G}_{K_1,K_2}^{n,k-1-\min\{i-1,k-i\},s}\right),
\end{equation}
since $\max\{i-1,k-i\}=k-1-\min\{i-1,k-i\}$. This means we have another set of positive integers $(n,k',i')=(n,k-1-\min\{i-1,k-i\},s)$ with $c\left(\oplus_{K_1,K_2}^{n,k'}\right)=c\left(\mathrm{G}_{K_1,K_2}^{n,k',i'}\right)$. 

If $i=k/m$, then $i-1<k-i$ and we find that $s=k/m$ and $k'=k-i=\tfrac{m-1}{m}k$ and hence $s=k'/(m-1)$. Repeating this process, we obtain $c\left(\oplus_{K_1,K_2}^{n,\tilde{k}}\right)=c\left(\mathrm{G}_{K_1,K_2}^{n,\tilde{k},\tilde{i}}\right)$ for some $\tilde{i}=\tilde{k}$, which by the remarks above implies that $c(K_1\# K_2)=c(K_1)+c(K_2)$.

If $i=\tfrac{m-1}{m}k+1$, then $i-1>k-i$ and we obtain $s=2i-k-1=\tfrac{m-2}{m}k+1$ and $k'=i-1=\tfrac{m-1}{m}k$. Therefore $s=\tfrac{m-2}{m-1}k'+1$. Repeating this process, we obtain $c\left(\oplus_{K_1,K_2}^{n,\tilde{k}}\right)=c\left(\mathrm{G}_{K_1,K_2}^{n,\tilde{k},\tilde{i}}\right)$ for some $\tilde{k}$ and $\tilde{i}=1$, which again implies $c(K_1\# K_2)=c(K_1)+c(K_2)$.
%Since $k-1-\min\{i-1,k-i\}<k$, we can repeat the process and must eventually end up in one of the first two cases of this proof. Thus in any case $c(K_1\# K_2)=c(K_1)+c(K_2)$.
\end{proof}

At the moment it seems unlikely that one could solve the crossing number conjecture by finding values for $n$, $k$ and $i$ for which the condition in Proposition \ref{rec} is satisfied. It is more promising to aim for a pure existence statement. This is of course highly speculative, but the hope is that the situation becomes similar to the one in Section \ref{sec:higher}, where it is very hard for a given $n$ to decide whether $c(\theta_{K_1,K_2}^n)=n(c(K_1)+c(K_2))$, but we know that if $n$ is large enough, then the equality is satisfied.

There are multiple other ways that one could extend the results outlined here to other types of graphs, all of which seem to give some inequalities and conditional results. It is a part of ongoing research, whether the results obtained by studying some of these graphs actually give us something new, something that we can not find by studying higher degree theta-curves.

Throughout this article we have worked under the assumption that $K_1$ and $K_2$ are prime. Many of the stated results remain true if we drop this assumption. Notably, for large enough $n$ the minimal crossing number of $\theta_{K_1,K_2}^n$ is equal to $n(c(K_1)+c(K_2))$. 
The definition of $\Omega_{K_1,K_2}^n$ has to be slightly adjusted. In particular, $x_i\cup x_j$ and $z_i\cup z_j$ are not allowed to be of the form $K\# K$, if $K$ is any summand of $K_1\# K_2\# K_1\# K_2$ other than $K_1$ or $K_2$. With this definition we again obtain that for large enough $n$ the crossing number satisfies $c(\Omega_{K_1,K_2}^n)=n(c(K_1)+c(K_2))$.

The results from Section \ref{sec:double} also remain largely true. Since the signs in the construction of $\tilde{D}$ can be chosen in such a way that $x_i\cup x_j$ and $z_i\cup z_j$ are always either a trefoil or the Whitehead double of a non-trivial knot (all of which have genus 1 and are therefore prime), $\tilde{D}$ is the diagram of a higher degree theta-curve in $c(\Omega_{K_1,K_2}^n)$. Thus we again have $c(K_1\# K_2)\geq \frac{1}{n^2} c(\Omega_{K_1,K_2}^n)$ for all $n$.

One difference in the setting of composite summands is that $\Gamma(\tilde{D})$ could have triangles even if $c(K_1\# K_2)\neq c(K_1)+c(K_2)$. Namely, there could be some prime summand $K$ of $K_1\# K_2$ such that $c(K_1\# K_2)=c(K)+c(K')$, where $K_1\# K_2=K\# K'$.

Thus any lower bound for $c(K_1\# K_2)$ that is obtained by finding $n$ such that $c(\Omega_{K_1,K_2}^n)=n(c(K_1)+c(K_2))$ does not relate the crossing number of a composite knot to the crossing numbers of its prime summands, but rather the crossing numbers of some decomposition into two summands, $K$ and $K'$.  

\ \\
\textbf{Acknowledgments:}\\
The author is grateful to Mark Dennis, Mikami Hirasawa, Jonathan Robbins, Kouki Taniyama and De Witt Sumners for valuable discussions and comments. This work is funded by the Leverhulme Trust Research Programme Grant RP2013- K-009, SPOCK: Scientific Properties Of Complex Knots.
\ \\
Conflict of interests: none.


\begin{thebibliography}{00}


\bibitem{diao} Y Diao. \textit{The additivity of the crossing number}. J Knot Theory Ramif \textbf{13}, 7 (2004), 857--866.

\bibitem{torus} H Gruber. \textit{Estimates for the minimal crossing number}. (arXiv:math/0303273) (2003).

\bibitem{kaufalt} L Kauffman. \textit{State models and the Jones polynomial}. Topology \textbf{26}:3 (1987), 395--407.

\bibitem{kauffman} L Kauffman, J Simon, K Wolcott, P Zhao. \textit{Invariants of theta-curves and other graphs in 3-space}. Topology and its Applications \textbf{49} (1993), 193--216.


\bibitem{lackenby} M Lackenby. \textit{The crossing number of composite knots}. Journal of Topology \textbf{2},4 (2009), 747--768.

\bibitem{lt88} W B R Lickorish, M B Thistlewaite. \textit{Some links with non-trivial polynomials and their crossing numbers}. Comment. Math. Helv. \textbf{63} (1988), 527--539.

\bibitem{malyutin} A Malyutin. \textit{On the question of genericity of hyperbolic knots}. (arXiv:1612.03368) (2016)

\bibitem{turaev} S Matveev, V Turaev. \textit{A semigroup of theta-curves in 3-manifolds}. Mosc Math J \textbf{11}, 4 (2011), 805--814.

\bibitem{murasugi} K Murasugi. \textit{Jones polynomials and classical conjectures in knot theory}. Topology \textbf{26}:2 (1987), 187--194.

\bibitem{thistle} M B Thistlewaite. \textit{A spanning tree expansion of the Jones polynomial}. Topology \textbf{26}:3 (1987), 297--309.

\bibitem{knotoids} V Turaev. \textit{Knotoids}. Osaka Journal of Mathematics \textbf{49}, 1 (2012), 195--223.


\end{thebibliography}
\end{document}